\documentclass[11pt,leqno]{article}
\usepackage{amsmath}
\usepackage{amssymb}
\usepackage{amsthm}
\usepackage{curves}
\usepackage{epic}
\usepackage{rotating}
\usepackage{epsf}
\usepackage[scanall]{psfrag}
\usepackage{graphicx}
\usepackage{esint} 
\usepackage{enumerate} 
\usepackage{pgf}
\usepackage{color}
\usepackage[top=5cm, bottom=1cm, left=3cm, right=3cm, heightrounded,
  marginparwidth=3cm, marginparsep=3mm]{geometry}
\usepackage[colorlinks,citecolor=blue,pagebackref,hypertexnames=false]{hyperref}
  
\definecolor{br}{RGB}{230,115,0}

\allowdisplaybreaks

\numberwithin{equation}{section}

\newtheorem{theorem}{Theorem}[section]
\newtheorem{prop}[theorem]{Proposition}
\newtheorem{corollary}[theorem]{Corollary}
\newtheorem{lemma}[theorem]{Lemma}

\newtheorem{defn}[theorem]{Definition}

\newtheorem*{remark}{Remark}

\renewenvironment{proof}{\medskip\noindent{\bf Proof.}}{\medskip}
\newenvironment{myproof}[1][\proofname]{
  \medskip \noindent {\bf #1}
}{\medskip}

\numberwithin{equation}{section}

\newcommand{\RR}{{\mathbb{R}}}

\newcommand{\dist}{\operatorname{dist}}

\newcommand{\R}{\mathcal{R}}

\renewcommand{\emptyset}{\mbox{\textup{\O}}}

\DeclareMathOperator{\supp}{supp}

\DeclareMathOperator{\diam}{diam}
\DeclareMathOperator{\interior}{int}

\DeclareMathOperator*{\Lip}{Lip}

\def\XXint#1#2#3{{\setbox0=\hbox{$#1{#2#3}{\int}$}
     \vcenter{\hbox{$#2#3$}}\kern-.5\wd0}}

\DeclareMathOperator{\divg}{div}
\DeclareMathOperator{\spt}{spt}

\DeclareMathOperator{\capacity}{cap_2}
\DeclareMathOperator{\capacityo}{cap_1}
\DeclareMathOperator{\capacityp}{cap_p}
\renewcommand{\R}{\mathbb{R}}
\newcommand{\NN}{\mathbb{N}}

\newcommand{\pO}{\partial\Omega}

\newcommand{\Oj}{\Omega_j}
\newcommand{\Oinf}{\Omega_\infty}
\newcommand{\pOj}{\partial\Omega_j}
\newcommand{\pOinf}{\partial\Omega_\infty}

\newcommand{\sj}{\sigma_j}
\newcommand{\sinf}{\mu_\infty}
\newcommand{\oj}{\omega_j}
\newcommand{\oinf}{\omega_\infty}
\newcommand{\mj}{\mu_j}
\newcommand{\minf}{\mu_\infty}

\newcommand{\Sinf}{\Sigma_\infty}

\newcommand{\ojA}{\omega_j^{A(p,r)}}

\newcommand{\wRj}{\widetilde m_j}

\newcommand\res{\hbox{ {\vrule height.22cm}{\leaders\hrule\hskip.2cm} } }

\begin{document}
\allowdisplaybreaks

\makeatletter{\renewcommand*{\@makefnmark}{}
\makeatother}

\title{Boundary rectifiability and elliptic operators with $W^{1,1}$ coefficients}

\author{Tatiana Toro\footnote{The first author was partially supported
by the Robert R. \& Elaine F. Phelps Professorship in Mathematics, the Craig McKibben \& Sarah Merner Professor in Mathematics and DMS grants 1361823 and 1664867.}  ~\&  Zihui Zhao\footnote{The second author was partially supported by DMS 
grants 1361823, 1500098 and 1664867. \newline  This material is based upon work supported by the National Science Foundation under DMS grant 1440140 while the authors were in residence at the Mathematical Sciences Research Institute in Berkeley, California, during the Spring 2017 semester.}
}
\date{}






\maketitle

\begin{abstract}
We consider second order divergence form elliptic operators with $W^{1,1}$ coefficients,
in a uniform domain $\Omega$ with Ahlfors regular boundary. We show that the $A_\infty$ property of the elliptic measure associated to any such operator implies that $\Omega$ is a set of locally finite perimeter whose boundary, $\pO$, is rectifiable. 
As a corollary we show that for this type of operators, absolute continuity of the surface measure with respect to the elliptic measure is enough to guarantee rectifiability of the boundary.
In the case that the coefficients are continuous we obtain additional information about $\Omega$.
\end{abstract}

\section{Introduction}

Recently there has been considerable activity seeking necessary and sufficient conditions on the geometry of a domain to ensure that the elliptic measure of a second order divergence form 
operator is absolutely continuous with respect to the surface measure of the boundary of the domain. In the case of the Laplacian contributions by several authors have led to a detailed 
understanding of the situation (see \cite{AAM}, \cite{ABHM}, \cite{AHMMMTV}, \cite{AHMNT}, \cite{BJ}, \cite{D1}, \cite{DJe}, \cite{GMT}, \cite{HLMN}, \cite{HM1}, \cite{HM4}
\cite{HMM}, \cite{HMU},  \cite{KT1}, \cite{KT2}, \cite{KT3}). In the variable coefficient case the \emph{forward} direction, which corresponds to the boundary regularity of the solutions in either 
Lipschitz or chord-arc domains is well understood. Examples constructed on \cite{CFK} and \cite{MM} 
showed that there exist second order divergence form elliptic operators on $C^1$ domains whose elliptic measure is singular with respect to the surface measure of the boundary of the 
domain. Thus it became clear that additional conditions on the matrix $A$ were needed. Different approaches emerged: perturbation theory 
(see \cite{D2}, \cite{E},\cite{F}, \cite{FKP}, \cite{MPT1}, \cite{MPT2}, \cite{CHM}); conditions on the structure or the oscillation of the matrix $A$ (see \cite{KP}, \cite{KKPT1}, \cite{HKMP}, 
\cite{KKPT2}); properties of the solutions with either bounded or $BMO$ boundary data (see \cite{DKP}, \cite{KKPT2}, \cite{Zh}).
In the variable coefficient case the \emph{free boundary regularity question}, which investigates the geometry of the domain under the assumption that the elliptic measure is well behaved with respect to 
the surface measure, is not well understood. In \cite{HMT1}, Hofmann, Martell and Toro proved that if $\Omega$ is a uniform domain with Ahlfors regular boundary, the matrix $A$ is locally Lipschitz and its 
gradient satisfies a Carleson condition then the $A_\infty$ property of the elliptic measure associated to any such operator implies that $\partial\Omega$ is uniformly rectifiable, that is 
$\Omega$ is an NTA domain. In \cite{ABHM}, Akman et al. obtained a qualitative version of these results. See also \cite{AGMT} for related work.

One of the main motivations of this paper is to understand whether the elliptic measure of divergence form operator distinguishes between a 
rectifiable and a purely unrectifiable boundary. Geometrically we consider bounded uniform domains $\Omega\subset \mathbb{R}^n$ with $n\ge 3$ (see Definition \ref{def:uniform}) with Ahlfors regular boundary (see Definition 
\ref{def:ADR}).
Analytically we consider 
second order divergence form elliptic and symmetric operators with $W^{1,1}(\Omega)\cap L^{\infty}(\Omega)$ or $C(\overline\Omega)$ coefficients whose elliptic measure is an $A_\infty$ 
weight in the sense of \cite{HMU} (see Definition \ref{def:AinftyHMU}) with respect to the surface measure $\sigma=\mathcal{H}^{n-1}\res\pO$. Our goal is to understand how the analytic 
information yields geometric insight on the domain and its boundary.
More precisely, let $\Omega\subset\mathbb{R}^n$ open and bounded, with $n\geq 3$. We consider uniformly elliptic divergence form operators $L=-\divg(A(X)\nabla)$, where $A(X) = \left( a_{ij}(X) \right)_{i,j=1}^n$ is a real symmetric matrix satisfying either $A\in W^{1,1}(\Omega) \cap L^{\infty}(\Omega)$ or $A\in C(\overline\Omega)$. We assume that there exist constants $0<\lambda<\Lambda<\infty$ such that $\|A\|_{L^\infty(\Omega)}\le \Lambda$ and for all $X\in\Omega$,
\begin{equation}\label{def:UE}
	\lambda |\xi|^2 \leq \langle A(X)\xi, \xi \rangle \leq \Lambda |\xi|^2, \qquad\text{for all } \xi\in\mathbb{R}^n\setminus\{0\}.
\end{equation}
Our main goal is to understand the extent to which the regularity of the elliptic measures of these operators determines the structure of the boundary. In particular we care about whether the absolute continuity (quantitative or qualitative) of surface measure with respect to elliptic measure ensures the exterior corkscrew property of the domain or the rectifiability of its boundary. Theorems \ref{thm:lfp}, \ref{thm:qc} and \ref{thm:lnta} provide answers to these queries.

\begin{theorem}\label{thm:lfp}
Let $\Omega\subset \mathbb{R}^n$ be a bounded uniform domain with Ahlfors regular boundary. Let $L = -\divg(A(X)\nabla)$ with $A\in W^{1,1}(\Omega)\cap L^{\infty}(\Omega)$ satisfying \eqref{def:UE}. Suppose that the elliptic measure $\omega_{L} \in A_{\infty}(\sigma)$ in the sense of \cite{HMU} (see Definition \ref{def:AinftyHMU}), then $\Omega$ is a set of locally finite perimeter, whose measure theoretic boundary coincides with its topological boundary $\mathcal{H}^{n-1}$-a.e. Thus $\pO$ is $(n-1)$-rectifiable.  
\end{theorem}

\begin{theorem}\label{thm:qc}
Let $\Omega\subset \mathbb{R}^n$ be a bounded uniform domain with Ahlfors regular boundary. Let $L = -\divg(A(X)\nabla)$ with $A\in W^{1,1}(\Omega)\cap L^{\infty}(\Omega)$ satisfying \eqref{def:UE}. Suppose $X_0\in\Omega$ is such that $\delta(X_0)\sim \diam\Omega$, and denote $\omega = \omega_L^{X_0}$. Then
if $\sigma \ll \omega$, $\pO$ is $(n-1)$-rectifiable.
\end{theorem}

Theorem \ref{thm:qc} should be understood as a corollary of Theorem \ref{thm:lfp}. In fact modulo a stopping time argument the proof can be reduced to applying a local version of Theorem \ref{thm:lfp}. 
By taking this approach we would like to emphasize the fact that, in this area, quantitative results yield qualitative ones. See section \ref{qualitative}.


\begin{theorem}\label{thm:lnta}
Let $\Omega\subset \mathbb{R}^n$ be a bounded uniform domain with Ahlfors regular boundary. Let $L = -\divg(A(X)\nabla)$ with $A\in C(\overline\Omega)$ satisfying \eqref{def:UE}. Suppose that the elliptic measure $\omega_{L} \in A_{\infty}(\sigma)$ in the sense of \cite{HMU} (see Definition \ref{def:AinftyHMU}), then there exists $r_\Omega>0$ such that $\Omega$ satisfies the exterior corkscrew condition for all $r<r_{\Omega}$. In particular $\Omega$ is an NTA domain.  
\end{theorem}

It is important to differentiate the result in Theorem \ref{thm:lnta} and those in \cite{HM1}, \cite{HMT1} and \cite{HMU}. The key difference is that although one shows that the domain is NTA, the constants are not ``uniform" in the sense that they do not depend only on the \emph{allowable} constants, namely the dimension $n$, the ellipticity constants of $A$, the Ahlfors constants, and the constants that determine the uniform character of the domain. Here $r_\Omega$ is obtained via compactness and there might depend a priori on the domain $\Omega$ itself.

After this paper was posted it was brought to our attention that Azzam and Mourgoglou \cite{AM} were also working in similar questions. Both sets of hypotheses are different and therefore so are the results. The techniques are drawn from harmonic analysis and geometric 
measure theory. These bodies of work complement each other well and contribute to a fast evolving field. 

We would like to thank the referee for the careful reading of the paper and very useful suggestions.


\section{Preliminaries}

In this section we provide a few definitions, properties of the Green function and elliptic measure, and some important properties of $W^{1,1}$ functions on uniform domains. From now on we simply use $\omega$ to denote the elliptic measure $\omega_L$, when there is no confusion as to what the underlying operator is.

\begin{defn}\label{def:ICC}
The domain $\Omega$ is said to satisfy the \textbf{(interior) corkscrew condition} (resp. exterior corkscrew condition) if  there are $M, R>0$ such that for any $q\in\pO$, $r\in (0,R)$, there exists a corkscrew point (or non-tangential point) $A=A(q,r) \in \Omega$ (resp. $A\in \Omega^c$, the complement of domain $\Omega$) such that
 \begin{equation}\label{eqn:nta-M}
 	|A-q|<r \quad \text{and}\quad \delta(A):= \dist(A,\partial\Omega) > \frac{r}{M}.
 \end{equation} 
\end{defn}

\begin{defn}\label{def:HCC}
The domain $\Omega$ is said to satisfy the \textbf{Harnack chain condition} if there are universal constants $C_1>1$ and $C_2>0$, such that for every pair of points $A$ and $A'$ in $\Omega$ satisfying
\[
	\Lambda:=\frac{|A-A'|}{\min\{\delta(A), \delta(A')\}} > 1, 
\] 
 there is a chain of open Harnack balls $B_1, B_2, \cdots, B_M$ in $\Omega$ that connects $A$ to $A'$. Namely, $A\in B_1$, $A'\in B_M$, $B_j\cap B_{j+1}\neq\emptyset$ and 
\begin{equation}\label{preHarnackball}
 	C_1^{-1}\diam(B_j) \leq \delta(B_j) \leq C_1\diam(B_j)
\end{equation}
     for all $j$.
    Moreover, the number of balls $M$ satisfies
\begin{equation}\label{eqn:length}
	M \leq  C_2(1+\log_2 \Lambda).
\end{equation}  
\end{defn}

\begin{defn}\label{def:uniform}
A domain $\Omega\subset \RR^n$ is said to be \textbf{uniform} if it satisfies: 
\newline
 (i) the interior corkscrew condition and \newline
(ii) the Harnack chain condition.
\end{defn}

\begin{defn}\label{def:nta}
	A uniform domain $\Omega\subset\RR^n$ is said to be \textbf{NTA} if it satisfies the exterior corkscrew condition.
\end{defn}

For any $q\in\pO$ and $r>0$, let $\Delta=\Delta(q,r)$ denote the surface ball $B(q,r) \cap \pO$.
We always assume $r< \diam \Omega$.

\begin{defn}\label{def:ADR}
We say that the boundary of $\Omega$ is \textbf{Ahlfors regular} if there exist constants $C_2>C_1>0$, such that for any $q\in\pO$ and any radius $r>0$,
\[ C_1 r^{n-1} \leq \sigma(\Delta(q, r)) \leq C_2 r^{n-1}, \]
where $\sigma = \mathcal{H}^{n-1}|_{\pO}$ is the surface measure.
\end{defn}

It was proved in \cite{Zh} (section 3) that a domain in $\mathbb{R}^n$ with $n\ge 3$ with Ahlfors regular boundary satisfies the so-called capacity density condition and is in particular Wiener regular.
\begin{defn}\label{def:CDC}
A domain $\Omega\subset \mathbb{R}^n$ with $n\ge 3$ is said to satisfy the \textbf{capacity density condition (CDC)} if there exist constants $c_0, R>0$ such that 
\begin{equation}\label{eqn:CDC}
	\capacity\left(\overline{B(q,r)} \cap \Omega^c \right) \geq c_0 r^{n-2},\quad \text{for any } q\in\pO \text{ and any } r\in (0,R).
\end{equation}	
For any compact set $K$, the $p$-capacity with $1\leq p<\infty$ is defined as follows:
	\begin{equation}\label{def:cap}
		\capacityp(K) = \inf \bigg\{\int |\nabla \varphi|^p dx: \varphi \in C_c^{\infty}(\mathbb{R}^n), K\subset \interior\{\varphi\geq 1\}\bigg\}.
	\end{equation}
\end{defn}

Moreover, if $\Omega$ is a uniform domain satisfying the CDC (in particular if $\pO$ is Ahlfors regular) and $L=-\divg(A\nabla)$ with $A\in L^{\infty}(\Omega)$ satisfying \eqref{def:UE}, the work of Gr\"{u}ter and Widman \cite{GW} guarantees the existence of a Green function. The work in \cite{HMT2} describes the behavior of the Green functions with respect to the elliptic measure. In particular the results proved in \cite{NTA} for harmonic functions on NTA domains extend to solutions of $L$ on uniform domains with the CDC. We summarize below the results which will be used later in this paper.

\begin{theorem}\label{thm:gw}
	Given an open bounded connected domain $\Omega\subset\RR^n$, there exists a unique function $G: \Omega \times \Omega \to \RR\cup \{\infty\}$, $G\geq 0$ such that the following hold:
	\begin{enumerate}
		\item for each $y\in\Omega$ and $r>0$, $G(\cdot, y) \in W^{1,2}(\Omega \setminus B(y,r))\cap W^{1,1}_0(\Omega)$;
		\item for all $\varphi \in C_c^{\infty}(\Omega)$, $\displaystyle\int \langle A\nabla G(x,y), \nabla \varphi(x) \rangle dx = \varphi(y)$;
		\item for each $y\in\Omega$, $G(\cdot,y)$ the Green function of $L$ in $\Omega$ with pole $y$, denoted by $G(x) = G(x,y)$, satisfies
	\begin{equation}\label{eqn:qweak}
		G\in L^{*}_{n/(n-2)}(\Omega) \text{ with } \|G\|_{L^{*}_{n/(n-2)}} \leq C(n,\lambda,\Lambda)
	\end{equation}
	\begin{equation}\label{eqn:qqweak}
		\nabla G \in L^*_{n/(n-1)}(\Omega) \text{ with } \|\nabla G\|_{L^{*}_{n/(n-1)}} \leq C(n,\lambda,\Lambda)
	\end{equation}
		\item For all $x,y\in\Omega$,
	\begin{equation}\label{eqn:gub}
		G(x,y) \leq C(n,\lambda,\Lambda) \frac{1}{|x-y|^{n-2}}
	\end{equation}
	\begin{equation}\label{eqn:glb}
		G(x,y) \geq C(n,\lambda,\Lambda) \frac{1}{|x-y|^{n-2}}, \quad\text{if } |x-y|\leq \frac{\delta(y)}{2}.
	\end{equation}
	\end{enumerate}
		
	Here $\delta(y)=\dist(y,\pO)$ and $L^*_p(\Omega)$ denotes the weak $L^p$ space.
\end{theorem}

\begin{defn}\label{weaklp}
	For $p>1$ we define the Banach space $L^*_p(\Omega)$ by 
	\[ L_p^*(\Omega) = \{f:\Omega \to \RR\cup\{\infty\}: f \text{ measurable and } \|f\|_{L^*_p(\Omega)} <\infty \} \]
	where $		\|f\|_{L^*_p(\Omega)} = \sup_{t>0} t \left |\{x\in\Omega: |f(x)|>t \}\right|^{\frac{1}{p}}$.
\end{defn}
Note that $\|f\|_{L_p^*(\Omega)} \leq \|f\|_{L^p(\Omega)}$,
and for $0<\epsilon\leq p-1$
\begin{equation}\label{eqn:lp*-lq}
	\|f\|_{L^{p-\epsilon}(\Omega)} \leq \left( \frac{p}{\epsilon} \right)^{\frac{1}{p-\epsilon}} |\Omega|^{\frac{\epsilon}{p(p-\epsilon)}} \|f\|_{L_p^*(\Omega)},
\end{equation}
where $|\Omega|$ denotes the Lebesgue measure of $\Omega$. Therefore \eqref{eqn:qweak}, \eqref{eqn:qqweak} and Theorem \ref{thm:gw} (1) imply that for any $1\leq q<n/(n-1)$ the Green function $G=G(\cdot, y)$ satisfies
	\begin{equation}\label{eqn:glq}
		G\in W_0^{1,q}(\Omega), \text{ with } \|G\|_{W_0^{1,q}(\Omega)} \leq C(n,\lambda,\Lambda,q, |\Omega|).
	\end{equation}

\noindent Recall that if a domain $\Omega$ is Wiener regular, by \cite{LSW} $\Omega$ is regular for $L$ as well. In this case there exists a family of probability measures $\{\omega^X\}_{X\in\Omega}$ such that if $Lu=0$ in $\Omega$ and $u=f$ on $\pO$ with $f\in C(\pO) \cap W^{1,2}(\Omega)$, then
\begin{equation}\label{eqn:ellipt-rep}
	u(X) = \int f(q) d\omega^X(q).
\end{equation}

The following results are proved in detail in \cite{HMT2}, the arguments are closely related to those in \cite{NTA}. When we refer to the allowable constants, we mean the dimension $n$, the ellipticity constants $\lambda$ and $\Lambda$, the $L^\infty(\Omega)$ norm of $A$, $c_0$ as in \eqref{eqn:CDC} and the constants that describe the uniform character of $\Omega$ which is $M$ as in \eqref{eqn:nta-M}, \eqref{preHarnackball} and \eqref{eqn:length}.

\begin{lemma}\label{lem:vanishing}
	Let $\Omega$ be a uniform domain satisfying the CDC. There exists $\beta>0$ (depending on the allowable constants) such that for $q\in\pO$ and $r<\diam \Omega$, and $u\geq 0$ with $Lu=0$ in $ B(q,2r) \cap\Omega$, if $u$ vanishes continuously on $\Delta(q,2r) =  B(q,2r) \cap \pO$, then
	\[
		u(X) \leq C\left( \frac{|X-q|}{r} \right)^{\beta} \sup_{B(q, 2r)\cap\Omega} u,
	\]
	where $C$ depends on the allowable constants.
\end{lemma}

\begin{corollary}\label{ellip-lb}
	Let $\Omega$ be a uniform domain satisfying the CDC. There exists $m_0>0$ depending on the allowable constants such that for any $q\in\pO$ and $r<\diam \Omega$,
	\[
		\omega^{A(q,r)} (\Delta(q,r)) \geq m_0.
	\]
	Here $A(q,r)$ denotes a non-tangential point for $q$ at radius $r$.
\end{corollary}

\begin{lemma}\label{harn-princ}
	Let $\Omega$ be a uniform domain satisfying the CDC. Let $q\in\pO$ and $r<\diam\Omega$. If $u\geq 0$ with $Lu=0$ in $\Omega\cap B(q,4r)$ and $u$ vanishes continuously on $\Delta(q,4r)$, then 
	\begin{equation}\label{eqn:1.3}
		u(X) \leq C u(A(q,r)), \quad \text{for } X\in\Omega\cap B(q,r).
	\end{equation}
	Here $C>0$ depends on the allowable constants.
\end{lemma}

\begin{lemma}\label{CFMS}
	Let $\Omega$ be a uniform domain satisfying the CDC. There exists $C>0$ depending on the allowable constants such that for $q\in\pO$ and $r<\diam \Omega/M$,
	\[
		C^{-1} \leq \dfrac{\omega^X(\Delta(q,r))}{r^{n-2}G(A(q,r),X)} \leq C, \quad \text{for any } X\in \Omega\setminus B(q,4r),
	\]
	where $G(\cdot,X)$ is the Green function of $L$ in $\Omega$ with pole $X$ as defined in Theorem \ref{thm:gw} and $\omega^X$ is the elliptic measure of $L$ with pole $X$ as in \eqref{eqn:ellipt-rep}.
\end{lemma}

\begin{lemma}\label{doubling}
	Let $\Omega$ be a uniform domain satisfying the CDC. Let $q\in\pO$ and $0<r<\diam\Omega/2M$, if $X\in \Omega\setminus B(q,2Mr)$, then for $s\in (0,r)$,
	\[
		\omega^X(\Delta(q,2s)) \leq C\omega^X(\Delta(q,s)),
	\]
	where $C$ only depends on the allowable constants.
\end{lemma}

We also need the following representation formula whose proof requires a number of approximation arguments which appear in detail in \cite{HMT2}.
\begin{prop}\label{representation}
	Let $\Omega$ be a uniform domain satisfying the CDC. Let $L=-\divg(A\nabla)$ with $A\in L^{\infty}(\Omega)$ symmetric satisfying \eqref{def:UE}. Let $G(\cdot,X_0)$ denote the Green function of $L$ in $\Omega$ with pole $X_0$ and $\omega^{X_0}$ the corresponding elliptic measure. Then for any $\varphi \in C_c^{\infty}(\RR^n)$,
	\[
		-\int_{\Omega} \langle A(X)\nabla G(X,X_0), \nabla \varphi(X) \rangle dX = \int_{\pO} \varphi d\omega^{X_0} -\varphi(X_0).
	\]
\end{prop}

In the statement of Theorems \ref{thm:lfp} and \ref{thm:lnta} we assume that $\omega_L \in A_{\infty}(\sigma)$ in the sense of \cite{HMU}. This is a uniform scale invariant notion of $A_{\infty}$-weight.

\begin{defn}\label{def:AinftyHMU}
	The elliptic measure of $\Omega$ is said to be of class $A_{\infty}$ with respect to the surface measure $\sigma= \mathcal{H}^{n-1}\res{\pO}$ in the sense of \cite{HMU}, which we denote by $\omega\in A_\infty(\sigma)$, if there exist 
	positive constants $C$ and $\theta$ such that for any surface ball $\Delta=B(q,r)\cap \pO$, with $r\le\diam \Omega$, the elliptic measure with pole at $A(q,r)$ satisfies
	   \begin{equation}\label{eqn:Ainfty}
		\frac{\omega^{A(q,r)}(E)}{\omega^{A(q,r)}(\Delta')} \leq C\left( \frac{\sigma(E)}{\sigma(\Delta')} \right)^{\theta},
	\end{equation}
where $A(q,r)\in \Omega$ is a non-tangential point for $q$ at radius $r$.We assume that \eqref{eqn:Ainfty} holds for any surface ball $\Delta'\subset \Delta$ and and Borel subset
$E\subset \Delta'$.
\end{defn}

In the next lemma we describe the properties of $W^{1,1}(\Omega)\cap L^{\infty}(\Omega)$ which are crucial to our arguments. Recall that uniform domains are $(\epsilon,\delta)$ domains in the language of Jones, see \cite{Jo}. Thus they are extension domains and in particular, if $A\in W^{1,1}(\Omega)$ there exists $\bar A\in W^{1,1}(\RR^n)$ such that $\bar A|_{\Omega} = A $ and $\|\bar A\|_{W^{1,1}(\RR^n)} \leq C\|A\|_{W^{1,1}(\Omega)}$, where $C$ depends on $n$ and the constants describing the uniform character of $\Omega$.

\begin{lemma}\label{A-vanishing-osc}
	Let $\Omega$ be a uniform domain with Ahlfors regular boundary. Let $A\in W^{1,1}(\Omega)$. Then for $\mathcal{H}^{n-1}$ a.e. $q\in\pO$ there exists a symmetric constant coefficient elliptic matrix $A^*(q)$ (with constants depending on the allowable constants) such that
	\begin{equation}\label{eqn:v-osc}
		\lim_{r\to 0} \left( \fint_{B(q,r)\cap\Omega} |A-A^*(q)|^{\frac{n}{n-1}} dX \right)^{\frac{n-1}{n}} = 0.
	\end{equation}
\end{lemma}

\begin{proof}
	By the previous remark, there exists $\bar A\in W^{1,1}(\RR^n)$ such that $\bar A|_{\Omega} = A$ and $\|\bar A\|_{W^{1,1}(\RR^n)} \leq C\|A\|_{W^{1,1}(\Omega)}$. By Theorem 1 section 4.8 in \cite{Evans} we have that if $\bar A\in W^{1,1}(\RR^n)$, then there exists a Borel set $E\subset\RR^n$ such that $\capacityo(E)=0$ (recall the definition in \eqref{def:cap} with $p=1$) and 
	\[
		\lim_{r\to 0} \fint_{B(x,r)} \bar A = A^*(x)
	\]
	exists for all $x\in\RR^n\setminus E$. In addition
	\begin{equation}\label{eqn:1.7}
		\lim_{r\to 0} \left( \fint_{B(x,r)} |\bar A - A^*(x) |^{\frac{n}{n-1}} dy \right)^{\frac{n-1}{n}} = 0, \quad\text{for all } x\in\RR^n\setminus E.
	\end{equation}
	Note that by Theorem 3 in section 5.6 in \cite{Evans} since $\capacityo(E) = 0$ then $\mathcal{H}^{n-1}(E) = 0$. Hence \eqref{eqn:1.7} holds for $\mathcal{H}^{n-1}$ a.e. $q\in\pO$. Since for every $q\in\pO$ and $0<r<\diam\Omega$ there exists $A(q,r)\in\Omega$ such that $B(A(q,r),r/M) \subset \Omega \cap B(q,r)$ (see \eqref{eqn:nta-M}), we have 
	\[ c_n \left(\frac{r}{M} \right)^n\leq  |\Omega\cap B(q,r)| \leq c_n r^n, \] 
	where $c_n$ denotes the volume of a unit ball in $\RR^n$. Thus for $q\in \pO\setminus E$
	\[
		\left( \fint_{B(q,r)\cap\Omega} |A-A^*(q)|^{\frac{n}{n-1}} dX \right)^{\frac{n-1}{n}} \leq C_{n,M} \left( \fint_{B(q,r)} |\bar A-A^*(q)|^{\frac{n}{n-1}} dX \right)^{\frac{n-1}{n}}
	\]
	because $\bar A|_{\Omega} = A$. Combined with \eqref{eqn:1.7} we get
	\[
		\lim_{r\to 0} \left( \fint_{B(q,r)\cap\Omega} |A-A^*(q)|^{\frac{n}{n-1}} dX \right)^{\frac{n-1}{n}} = 0,
	\]
	and moreover, H\"older inequality gives
	\begin{align*}
		\lim_{r\to 0} \left| \fint_{B(q,r)\cap\Omega} A dX - A^*(q) \right| & \leq \lim_{r\to 0} \fint_{B(q,r)\cap\Omega} |A-A^*(q)| dX \nonumber \\
		& \leq \lim_{r\to 0} \left( \fint_{B(q,r)\cap\Omega} |A-A^*(q)|^{\frac{n}{n-1}} dX \right)^{\frac{n-1}{n}} = 0 
	\end{align*}
	for $q\notin E$. Since \eqref{def:UE} holds for any $\xi\in\RR^n\setminus\{0\}$, we have
	\[
		\lambda|\xi|^2 \leq \left\langle \left( \fint_{B(q,r)\cap\Omega} A \right) \xi, \xi \right\rangle \leq \fint_{B(q,r)\cap\Omega} \langle A\xi,\xi \rangle \leq \Lambda|\xi|^2.
	\]
	Letting $r$ tend to $0$ we conclude that
	\begin{equation}\label{eqn:1.12}
		\lambda|\xi|^2 \leq \langle A^*(q)\xi,\xi \rangle \leq \Lambda |\xi|^2.
	\end{equation}
	Since $A$ is symmetric, so is $\fint_{B(q,r)\cap\Omega} A$ for every $q$ and $r>0$. Moreover for $A\in L^{\infty}(\Omega)$, $\fint_{B(q,r)\cap\Omega} A$ is uniformly bounded in $q$ and $r$. Thus $A^*(q)$ is a real uniformly elliptic symmetric matrix and \eqref{eqn:v-osc} holds for $\mathcal{H}^{n-1}$ a.e. $q\in\pO$.
\end{proof}

%

\begin{defn}\label{def:cvHd}
	Let $\{ \mathcal{D}_j\}_j$ be a sequence of non-empty closed subsets of $\RR^n$. We say $\mathcal{D}_j$ converge to a closed set $\mathcal{D}_{\infty} $ in the Hausdorff distance sense and write $\mathcal{D}_j \to \mathcal{D}_{\infty}$, if their Hausdorff distance
	\[ D[\mathcal{D}_j, \mathcal{D}_{\infty} ] \to 0 \quad \text{as} \quad j\to\infty. \]
	Here 
	\[ D[E, F ] :=\max\left\{ \sup_{x\in E} \inf_{y\in F} |x-y|, \sup_{y\in F} \inf_{x\in E} |x-y| \right\}  \]
	is called the Hausdorff distance between two non-empty closed subsets $E, F$ of $\RR^n$.
\end{defn}

\begin{defn}\label{def:wcvm}
	Let $\{\mu_j\}$ be a sequence of Radon measures on $\RR^n$. We say $\mu_j$ converge weakly to a Radon measure $\mu_{\infty}$ and write $\mu_j \rightharpoonup \mu_{\infty}$, if 
	\[ \int f \mu_j \to \int f \mu_{\infty} \]
	for any bounded continuous function $f$.
\end{defn}

We finish this section by stating a compactness type lemma for Radon measures which are uniformly doubling and ``bounded below".
\begin{lemma}\label{lm:sptcv}
	Let $\{\mu_j\}_j$ be a sequence of Radon measures. Let $C_1, C_2>0$ be fixed constants. Assume
	\begin{enumerate}[i)]
		\item $0\in \spt\mj$ and $\mj(B(0,1)) \geq C_1$ for all $j$,
		\item For all $j\in\NN$, $q\in \spt \mj$ and $r>0$,
			\begin{equation}\label{unif-doubling}
				\mu_j(B(q,2r))\le C_2\mu_j(B(q,r))
			\end{equation}
	\end{enumerate}
	If there exists a Radon measure $\minf$ such that $\mj \rightharpoonup \minf$, then $\mu_\infty$ is doubling and
	\begin{equation}\label{eqn:spt-conv}
	 \spt \mj \to \spt \minf, 
	 \end{equation}
	in the Hausdorff distance sense uniformly on compact sets.
	Recall that if $\mu$ is a Radon measure, then $\spt\mu = \overline{\{ x\in\RR^n: \mu(B(x,r)) > 0 \text{ for any }r>0 \}}$.
	\end{lemma}

\begin{proof}
Since $0\in \spt\mj$ for all $j$, given any subsequence of $\mj$ there exists a further subsequence $\mu_{j_k}$ and a closed set $\Sinf$ such that $\spt\mu_{j_k} \to \Sigma_\infty$
in the Hausdorff distance sense uniformly on compact sets. For $x\in\Sinf$ there exist $x_{j_k}\in \spt\mu_{j_k}\cap \overline {B(x,1})$ such that $x_{j_k} \to x$. If $x \notin \spt \minf$ there is $r\in(0,1)$ such that $B(x,r)\cap \spt\minf = \emptyset$. Let $\varphi\in C_c^{\infty}(\R^n)$ such that $\varphi \equiv 1$ on $B(x,r/2)$, and $\supp\varphi \subset B(x,r)$. In particular we have $\supp\varphi \cap \spt\minf = \emptyset$.
For $k$ large enough, we also have $\varphi \equiv 1$ on $B(x_{j_k}, r/4)$. Hence
	\begin{equation}\label{eq:tmp2}
		\mu_{j_k}(B(x_{j_k}, r/4)) \leq \int\varphi d\mu_{j_k} \to \int\varphi d\minf = 0.  
	\end{equation} 
	Since $\{x_{j_k}\}$ is a bounded sequence in $\overline{B(x,1)}$, there is $l\in\mathbb{N}$ such that $|x_{j_k}| < 2^l$ for all $j_k$. Then we have $ B(0,1) \subset B(x_{j_k}, 2^{l+1})$. 
	Let $m\in \mathbb{Z}$ such that $2^{-m} \leq r < 2^{-m+1}$, then we have
	\begin{eqnarray*}\label{eqn-prelm2}
	 \mu_{j_k}(B(x_{j_k}, r/4)) &\geq & \mu_{j_k}(B(x_{j_k},2^{-m-2}))\ge
	 C_2^{-(m+l+3)} \mu_{j_k}(B(x_{j_k},2^{l+1}))\nonumber\\
	 & \ge &C_2^{-(m+l+3)} \mu_{j_k}(B(0,1)) \geq C_1  C_2^{-(m+l+3)},
	 \end{eqnarray*} 
	which contradicts \eqref{eq:tmp2}. Thus $\Sinf \subset \spt\minf$. Notice that this shows that any subsequential limit of $\spt\mu_{j}$ is included in $\spt\minf$.	
	
	On the other hand, if $y\in \spt\minf$, $r>0$ and $\{j_k\}$ is the subsequence
	above we have
	\[
	0< \minf(B(y,r)) \leq \liminf_{j_k\to\infty} \mu_{j_k}(B(y,r)). 
	\]
	For $r=1$ there exists $\tilde j_1$ such that for $j_k\ge \tilde j_1$, $\mu_{j_k}(B(y,1))>0$. Thus there is $y_1\in B(y,1)\cap\spt\mu_{\tilde j_1}$. Iteration guarantees that for each $k\in\mathbb{N}$ there exist 
	$\tilde j_k>\tilde j_{k-1}$ and $y_k\in B(y, 2^{-k})\cap \spt\mu_{\tilde j_k}$. Thus $y_k\to y$ as $k\to\infty$. Furthermore since $\spt\mu_{j_k} \to \Sigma_\infty$ then $y\in\Sigma_\infty$,
	thus $\spt\minf\subset \Sigma_\infty$. Therefore $\Sigma_\infty= \spt\minf$ which shows \eqref{eqn:spt-conv}.
	
	To show that $\minf$ is doubling let $x\in\spt \minf$ and $r>0$. There exist $x_j \in \spt\mj$ such that $ x_j \to x $. Thus for $j$ large enough $|x_j - x|<r/4$. Since $\mj \rightharpoonup \minf$ we have
	\begin{align*}
		\minf(B(x,2r)) & \leq \liminf_j \mj(B(x,2r)) \nonumber \\
		&  \leq \liminf_{j} \mj(B(x_j,3r)) \nonumber \\
		& \leq C_2^3 \liminf_j \mj \left(B\left(x_j, \frac{3}{8}r\right)\right) \nonumber \\
		& \leq C_2^3 \limsup_j \mj \left( \overline{B\left(x, \frac{5}{8}r \right)} \right) \nonumber \\
		& \leq C_2^3 \minf\left(\overline{B\left(x, \frac{5}{8} r\right)}  \right) \nonumber \\
		& \leq C_2^3 \minf(B(x,r)). \label{eqn:1.13}
	\end{align*}
	\end{proof}

\section{Blow-up and pseudo blow-up domains}\label{sect:blowup}

In this section we consider tangent objects as a way to understand the fine structure of $\pO$, under the assumptions of Theorem \ref{thm:lfp} and \ref{thm:lnta}. This requires looking at the tangent and pseudo-tangent domains and the corresponding functions and measures obtained via a blow-up. While tangent objects provide pointwise infinitesimal information, pseudo-tangents provide ``uniform infinitesimal" information.
The key point is to observe that the blow-ups or pseudo blow-ups of the operators satisfying the hypotheses of Theorems \ref{thm:lfp} and \ref{thm:lnta} lead to a constant coefficient operators. Our goal is to show that under these assumptions the tangent and pseudo-tangent objects satisfy the hypothesis of Theorem \ref{thm:hmu}, which is a simple generalization of Theorem 1.23 in \cite{HMU}. The details of its proof can also be found in \cite{HMT1}.  
\begin{theorem}\label{thm:hmu}
Let $\mathcal{D}$ be a uniform domain (bounded or unbounded) with Ahlfors regular boundary. Let $L$ be a symmetric second order elliptic divergence form operator with constant coefficient.
Assume that the elliptic measure $\omega\in A_\infty(\sigma)$ in the sense of \cite{HMU} (see Definition \ref{def:AinftyHMU}),
 then $\partial\mathcal{D}$ is uniformly rectifiable.
\end{theorem}
Getting to the point where we can apply this Theorem requires showing first that if $\Oinf$ is a blow-up or pseudo blow-up of $\Omega$, then $\Oinf$ is an unbounded uniform domain with Ahlfors regular boundary. To accomplish this we also need to blow up the Green function. Moreover the blow-up limit of the given elliptic operator, denoted by $L_{\infty}$, has constant coefficient. Once we have this, to show that $\omega_{L_\infty} \in A_{\infty}(\sigma_{\infty})$ for the blow-up domain and the limiting operator, we need to construct the elliptic measure $\omega_{L_\infty}^{Z} $ for any $Z\in\Oinf$ as a limiting measure compatible with the initial blow-up.

Let $X_0\in \Omega$ and $L=-\divg(A\nabla)$. Let
$G(X_0,\cdot)$ be the Green function for $L$ with pole $X_0$ and $\omega=\omega^{X_0}_{L}$ the corresponding elliptic measure. 
For $j\in \NN$, let $q_j\in\partial \Omega$ and $r_j>0$ such that $q_j\to q \in\pO$ and $r_j\to 0$. In some cases we assume that $q_j = q$ for all $j\in\NN$. We now consider
\begin{equation}\label{eqn:blow1}
\Omega_j=\frac{1}{r_j}(\Omega-q_j) \qquad\hbox{  and  } \qquad\partial\Omega_j=\frac{1}{r_j}(\partial\Omega-q_j).
\end{equation}
\begin{equation}\label{eqn:blow2}
u_j(Z) = \frac{ r_j^{n-2} G(X_0, q_j + r_j Z) }{ \omega(B(q_j,r_j))} \quad \text{for~} Z\in \Omega_j \hbox{    and    } u_j=0 \hbox{  in  } \Omega_j^c.
\end{equation}
\begin{equation}\label{eqn:blow3}
 \sigma_j(E) = \frac{\sigma(q_j + r_j E)}{r_j^{n-1}} \qquad\hbox{  and  }\qquad \omega_j (E) = \frac{\omega(q_j + r_j E)}{\omega(B(q_j,r_j))}.
\end{equation}

We follow the following conventions:
\begin{itemize}
	\item For $X\in\Omega$ we denote $\delta(X) = \dist(X,\pO)$ and for $Z\in \Oj$ we denote $\delta_j(Z) = \dist(Z,\pOj)$.
	\item For any $q\in\pO$ and $r\in (0, \diam  \pO)$, we use $A(q,r)$ to denote a non-tangential point in $\Omega$ with respect to $q$ at radius $r$, i.e.
		\[
		|A(q,r)-q| < r, \text{~and~} \delta(A(q,r)) \geq \frac{r}{M}. 
		\]		
	\item If $X\in\overline\Oj$ we denote by $\widetilde X=q_j + r_j X\in \overline\Omega$ .
	\item For any $p\in\pOj$ and $r\in(0,\diam\pOj)$, we use
		\[
		A_j(p,r) = \frac{A(\widetilde p,rr_j)-q_j} {r_j}
		\]
		as a non-tangential point in $\Oj$ with respect to $p$ at radius $r$. Here $\widetilde p=q_j+r_j p$.

\end{itemize}

Note that, modulo a constant, $u_j$ is the Green function for the operator $L_j = -\divg(A_j\nabla)$ with $A_j(Z)=A(r_jZ+q_j)$ in $\Oj$ with pole $X_j = (X_0-q_j)/r_j$.
Moreover $\oj$ is the corresponding elliptic measure with $0\in \spt\oj$ and $\oj(B(0,1)) = 1$.
 If $p\in\partial\Omega_j$ and $r\in (0,\diam \pOj)$ then $\widetilde p=r_j p +q_j\in\pO$,
\begin{equation}\label{eqn:blow-double}
\oj(B(p,2r))=\frac{\omega(B(\widetilde p, 2r r_j))}{\omega(B(q_j,r_j))}\le C\frac{\omega(B(\widetilde p, r r_j))}{\omega(B(q_j,r_j))}=C\oj(B(p,r))
\end{equation}
and
\begin{equation}\label{eqn:blow-sigma}
\frac{\sigma_j(B(p,r))}{r^{n-1}}= \frac{\sigma(B(\widetilde p, r r_j))}{(r r_j)^{n-1}}\sim 1
\end{equation}
Note also that $0\in\spt\sj$, $\sj(B(0,1)) = \sigma(B(q_j,r_j))/r_j^{n-1} \sim 1$. Hence $\{\sj\}$ and $\{\oj\}$ satisfy conditions i) and ii) of Lemma \ref{lm:sptcv}.  The three theorems below describe what happens as we let $j$ tend to infinity in the sequences defined in \eqref{eqn:blow1}, \eqref{eqn:blow2} and \eqref{eqn:blow3}.
 
\begin{theorem}\label{thm:pseudo-blow-geo}
\leavevmode
Let $\Omega\subset \R^n$ be a uniform domain with Ahlfors regular boundary. Let $L_A = -\divg(A(X)\nabla)$ be a divergence form uniformly elliptic operator in $\Omega$, assume that 
$A\in L^\infty(\Omega)$.  
Using the notation above, modulo passing to a subsequence  (which we relabel) we conclude the following
\begin{enumerate}
	\item \label{im:fcv} There is a function $u_\infty\in C(\R^n)$ such that $u_j \to u_\infty$ uniformly on compact sets. Moreover $\nabla u_j\rightharpoonup\nabla u_\infty$ in $L^2_{loc}(\R^n)$.
	 \item \label{im:dcv} Let $\Oinf=\{u_\infty >0\}$, then $\overline \Oj\to\overline \Omega_\infty$ and $\pOj\to \pOinf$ in the Hausdorff distance sense uniformly on compact sets.
	\item \label{im:uniform} $\Oinf$ is a non-trivial unbounded uniform domain.
	\item \label{im:hmcv} There is a doubling measure $\omega_{\infty}$ such that $\omega_j\rightharpoonup \omega_\infty$. Moreover $\spt\omega_\infty=\pOinf$.
	\item \label{im:smcv} There is an Ahlfors regular measure $\mu_\infty$ such that $\sj \rightharpoonup \mu_\infty$. Moreover $\spt\mu_\infty=\pOinf$. In particular this implies that 
	$\mu_\infty\ll \sigma_{\infty} := \mathcal{H}^{n-1}\res{\pOinf} \ll\mu_\infty$.
	
\end{enumerate}	

\begin{defn} \label{defn:tan} The domain $\Oinf$ is a pseudo-tangent domain to $\Omega$ at $q$. The function $u_\infty$ is a pseudo-tangent function to $G(X_0,\cdot)$ at $q$.
The measures $\mu_\infty$ and $\oinf$ are pseudo-tangent measures to $\sj$ and $\oj$ at $q$ respectively. If $q_j=q$ for all $j$ then $\Oinf$, $u_\infty$, $\mu_\infty$ and $\oinf$ are called 
tangents at $q$.
\end{defn}	

\begin{theorem}\label{thm:pseudo-blow-ana}
\leavevmode
Let $\Omega\subset \R^n$ be a uniform domain with Ahlfors regular boundary. Let $L = -\divg(A(X)\nabla)$ be a divergence form uniformly elliptic operator in $\Omega$. Assume that 
$A\in C(\overline\Omega)$. Then using the notation in Theorem \ref{thm:pseudo-blow-geo} we have that the function $u_\infty$ satisfies 
	\begin{equation}\label{limit-eq}
	 \left\{ \begin{array}{rl}
		-\divg(A(q)\nabla u_\infty) = 0 & \text{in~} \Oinf \\
		u_\infty > 0 & \text{in~}\Oinf \\
		u_\infty = 0 & \text{on~} \pOinf.
	\end{array} \right.  \end{equation}
	 i.e. $u_\infty$ is a Green function in $\Oinf$ for $L_{\infty}=\divg(A(q)\nabla)$ with pole at $\infty$.
	
Furthermore $\oinf$ is the harmonic measure corresponding to $u_\infty$, in the sense that
		\begin{equation}\label{limit-int}
		- \int_{\Oinf} A(q) \nabla u_\infty \cdot \nabla \psi dZ = \int_{\pOinf} \psi d\oinf, \quad\text{ for all } \psi\in C_c^{\infty}(\R^n). 
		\end{equation}

		\end{theorem}
		
\end{theorem}

\begin{theorem}\label{thm:AW11}
Let $\Omega \subset \R^n$ be a uniform domain with Ahlfors regular boundary. Let $L = -\divg(A(X)\nabla)$ be a divergence form uniformly elliptic operator in $\Omega$. Assume that 
$A\in L^{\infty}(\Omega) \cap W^{1,1}(\Omega)$. Then for $\mathcal{H}^{n-1}$ a.e. $q\in\pO$, using the notation in Theorem \ref{thm:pseudo-blow-geo}, under the assumption that $q_j=q$ for every $j$ we have that the corresponding function $u_{\infty}$ satisfies	\begin{equation}\label{limit-eq2}
	 \left\{ \begin{array}{rl}
		-\divg(A^*(q) \nabla u_\infty) = 0 & \text{in~} \Oinf \\
		u_\infty > 0 & \text{in~}\Oinf \\
		u_\infty = 0 & \text{on~} \pOinf.
	\end{array} \right.  \end{equation}
	 i.e. $u_\infty$ is a Green function in $\Oinf$ for $L_\infty=-\divg(A^*(q)\nabla)$ with pole at $\infty$.
	Furthermore $\oinf$ is the harmonic measure corresponding to $u_\infty$, in the sense that
		\begin{equation}\label{limit-int2}
		- \int_{\Oinf} A^*(q) \nabla u_\infty \cdot \nabla \psi dZ = \int_{\pOinf} \psi d\oinf, \quad\text{ for all } \psi\in C_c^{\infty}(\R^n). 
		\end{equation}
Here $A^*(q)$ is obtained as in Lemma \ref{A-vanishing-osc}.

\end{theorem}

\begin{myproof}[Proof of (\ref{im:fcv}) in Theorem \ref{thm:pseudo-blow-geo}]
	Fix $R>1$, for $j \geq j_0$ large enough we may assume $X_0 \in \Omega \setminus B(q_j, 2r_jR)$. For such $j$, $L_{j}u_j=0$ in ${B(0,2R)}\cap\Omega_j$. Here $L_j =-\divg(A_j\nabla) $ with $A_j(Z)=A(r_jZ+q_j)$. Note that $0\in\partial\Oj$ and $ \left(A(q_j, r_j) - q_j\right)/r_j \in \Oj$ is a non-tangential point for $0$ at radius $1$ for $\Oj$, we denote it by $A_j(0,1)$. Moreover
			\begin{equation}\label{eq:pjzj}
				u_j(A_j(0,1)) = \frac{r_j^{n-2} G(X_0, q_j+r_j A_j(0,1))}{\omega(B(q_j,r_j))} = \frac{r_j^{n-2} G(X_0, A(q_j,r_j))}{\omega(B(q_j,r_j))} \sim 1.  
			\end{equation} 
			 Let $A_j(0,R)\in\Oj$ denote a non-tangential point to $0$ at radius $R$, then by Harnack's inequality we have
			\[
			u_j(A_j(0,R)) \leq C(R) u_j(A_j(0,1)) \leq C'(R). 
			\] 
			Thus for any $Z\in \Oj\cap B(0,R)$, using Lemma \ref{harn-princ} we have
			\[
			u_j(Z) \leq C u_j(A_j(0,R)) \leq C(R). 
			\]
			Extending $u_j = 0$ on $ \Omega_j^c$ we conclude that the sequence $\{u_j\}_{j\geq j_0}$ is uniformly bounded in $\overline{B(0,R)}$. Since for each $j$, $L_j$ has ellipticity constants bounded below by $\lambda$ and above by $\Lambda$, $\|A_j\|_{L^{\infty}(\Oj)} = \|A\|_{L^\infty(\Omega)}$ and $\Oj$ is uniform and satisfies the CDC (as $\partial\Oj$ is Ahlfors regular) with the same constants as $\Omega$, then combining Lemma \ref{lem:vanishing} with DeGiorgi-Nash-Moser we conclude that the sequence $\{u_j\}_j$ is equicontinuous on $\overline{B(0,R)}$ (in fact uniformly H\"older continuous with the same exponent). 
			Using Arzel\`a-Ascoli combined with a diagonal argument applied on a sequence of balls with radii going to infinity, we produce a subsequence (which we relabel) such that $u_j \to u_\infty$ uniformly on compact sets of $\RR^n$. Note that the boundary Cacciopoli inequality yields
			\begin{align}
				\int_{B(0,R)}|\nabla u_j|^2\, dZ &=\int_{B(q_j, Rr_j)}\frac{r_j^{n-2}}{(\omega(B(q_j,r_j)))^2}|\nabla G(X_0,Y)|^2\, dY\nonumber\\
				&\leq C\frac{r_j^{n-4}}{(\omega(B(q_j,r_j)))^2}\int_{B(q_j, 2Rr_j)} G(X_0,Y)^2 dY. \label{eqn:3.1A}
			\end{align}
			Applying the boundary Harnack principle (see \eqref{eqn:1.3}) and Harnack inequality to estimate $G(X_0,A(q_j,2R r_j))$ by $G(X_0,A(q_j,r_j))$ and noting that $A(q_j,2R r_j)$ can be joined to $A(q_j,r_j)$ by a chain of length independent of $j$, \eqref{eqn:3.1A} becomes
			\begin{align}
				\int_{B(0,R)}|\nabla u_j|^2\, dZ & \leq C\frac{r_j^{2n-4}}{(\omega(B(q_j,r_j)))^2} G(X_0,A(q_j, 2Rr_j))^2 \nonumber\\
				& \leq C\frac{r_j^{2n-4}}{(\omega(B(q_j,r_j)))^2} G(X_0,A(q_j, r_j))^2. \label{eqn:3.1}
			\end{align}
			Finally applying Lemma \ref{CFMS} in \eqref{eqn:3.1} yields
			\[
				\sup_j \int_{B(0,R)}|\nabla u_j|^2 \leq C_R <\infty.
			\]
Recalling that the functions $\{u_j\}$ are uniformly bounded in $\overline{B(0,R)}$, we conclude that
\begin{equation}\label{eqn:3.4}
	\sup_j \|u_j\|_{W^{1,2}(B(0,R))} \leq C'_R <\infty.
\end{equation}		
Thus there exists a subsequence (which we relabel) which converges weakly in $W^{1,2}_{loc}(\RR^n)$. A standard argument allows us to conclude that $u_j \to u_{\infty}$ in $L^2_{loc}(\RR^n)$ and $\nabla u_j \rightharpoonup \nabla u_{\infty} $ in $L^2_{loc}(\RR^n)$. This shows \emph{(1)}.	
\end{myproof}
		
\begin{myproof}[Proof of (\ref{im:dcv}) in Theorem \ref{thm:pseudo-blow-geo}]
			Let $\Oinf = \{u_\infty>0\}$.  
			Since $0\in\pOj$ for all $j$, modulo passing to a subsequence (which we relabel) we have 
			\[
			\overline\Omega_j\to \Gamma_\infty\hbox{     and    }
			\pOj\to \Lambda_\infty\hbox{  as  }  j\to\infty.
			\] 
			Here $\Gamma_\infty$ and $\Lambda_\infty$ are closed sets, and the convergence is in the Hausdorff distance sense uniformly on compact sets. \newline
			
			\textbf{Claim}: $\Lambda_\infty = \pOinf$ and $\Gamma_\infty=\overline\Omega_\infty$.\newline

			Let $p\in\Lambda_\infty$, there is a sequence $p_j\in\pOj$ such that $\lim_{j\to\infty} p_j = p$. Note that
			$ u_\infty(p) =\lim_{j\to \infty}\ u_j(p)$. On the other hand since the $u_j$'s are uniformly H\"older continuous on compact sets $|u_j(p)|=|u_j(p)-u_j(p_j)| \leq C|p-p_j|^\alpha$, thus $\lim_{j\to \infty} u_j(p) = \lim_{j\to\infty} u_j(p_j) = 0$, and therefore $u_{\infty}(p) = 0$, i.e. $p\in \Omega_\infty^c $. Assume that there exists $\epsilon\in(0,1)$ such that $B(p,\epsilon) \subset \Oinf^c$, i.e. $u_\infty\equiv 0$ on $B(p,\epsilon)$. 
			Note that if $\widetilde p_j=q_j +r_j p_j$ then for $j$ large enough
			\[
				\left|A\left(\widetilde p_j,\frac{\epsilon}{2}r_j\right) - A(q_j,r_j)\right|  \le \frac{\epsilon}{2}r_j + |\widetilde p_j - q_j| + r_j 
				 \le\left(\frac{\epsilon}{2} + |p_j| + 1\right) r_j\le 2\left(|p| + 1\right) r_j
			\]
			and
			\[
			 \delta\left( A\left(\widetilde p_j,\frac{\epsilon}{2}r_j\right) \right) \geq \frac{1}{M} \frac{\epsilon }{2}r_j, \quad \delta\left( A(q_j,r_j)\right) \geq \frac{r_j}{M}. 
			 \]
			Applying Harnack inequality in $\Omega$, we know there is a constant $C=C(\epsilon, |p|)$ such that 
			\[
			G\left( X_0,A\left(\widetilde p_j,\frac{\epsilon }{2}r_j \right) \right) \geq C G(X_0,A(q_j,r_j)). 
			\]
			Recalling that $A_j(p_j,\frac{\epsilon}{2}) = \left(A\left(\widetilde p_j,\frac{\epsilon }{2}r_j \right) - q_j  \right)/r_j $, we have
			\begin{align}\label{eq:tmplb}
				u_j\left(A_j\left(p_j,\frac{\epsilon}{2}\right)\right)  &= \dfrac{ r_j^{n-2}G\left( X_0,A\left(\widetilde p_j,\frac{\epsilon }{2}r_j \right)\right)}{\omega(B(q_j,r_j))} \\
				& \geq C \dfrac{ r_j^{n-2}G\left( X_0,A(q_j,r_j) \right)}{\omega(B(q_j,r_j))}  
				 = Cu_j(A_j(0,1))
				 \geq C' >0,  \nonumber
			\end{align}
			where the constant $C'$ is independent of $j$.
			However, since for $j$ large enough
			\begin{equation} \label{eqn104}
			A_j\left(p_j,\frac{\epsilon}{2}\right) \in B\left(p_j,\frac{3\epsilon}{4}\right) \subset B(p,\epsilon),
			 \end{equation}
			the lower bound \eqref{eq:tmplb} combined with \eqref{eqn104}
			implies that $u_\infty \not\equiv 0$ on $B(p,\epsilon)$ which contradicts our assumption. Therefore $p\in\pOinf$, and $\Lambda_\infty \subset \pOinf$.
			
			To show that $\pOinf \subset \Lambda_{\infty}$, we assume that $p\notin \Lambda_{\infty}$, thus since $\Lambda_{\infty}$ is a closed set, there exists $\epsilon>0$ such that $B(p,2\epsilon)\cap\Lambda_{\infty} = \emptyset$. Since $\Lambda_{\infty}$ is the Hausdorff limit of $\pOj$ we have that for $j$ large enough $B(p,\epsilon) \cap \pOj=\emptyset$. Hence either $B(p,\epsilon) \subset \Oj$ or $B(p,\epsilon) \subset \interior \Oj^c $. If $B(p,\epsilon) \subset \Oj$ then $B(q_j + p r_j, \epsilon r_j) \subset \Omega$. Hence $\delta(q_j+p r_j)  >\epsilon r_j$ and $|A(q_j,r_j) - (q_j + p r_j)| \leq r_j (1+|p|)$. Thus there exists a Harnack chain joining $A(q_j,r_j)$ to $\left(q_j+p r_j\right)$ of length independent of $j$ and depending on $\epsilon$ and $|p|$. By Harnack's inequality $G(X_0,q_j+p r_j) \sim G(X_0,A(q_j,r_j))$ which combined with  \eqref{eq:pjzj} yields
			\begin{equation}\label{eqn:3.5}
				u_j(p) \sim u_j\left( \frac{A(q_j,r_j)-q_j}{r_j} \right) \sim u_j(A_j(0,1)) \sim 1.
			\end{equation}
			Hence for $X\in B(p,\epsilon)\subset \Oj$, again by Harnack inequality and \eqref{eqn:3.5} we have $u_j(X) \sim u_j(p) \sim 1$ with constants independent of $j$. Letting $j\to\infty$ we have that $u_{\infty}(X) \sim 1$ for $X\in B(p,\epsilon/2)$. Thus $B(p,\epsilon/2) \subset \Oinf = \{u_{\infty} >0\}$ and $p\notin\partial\Oinf$. If $B(p,\epsilon) \subset \interior \Oj^c$, then $u_j(X) = 0$ for all $X\in B(p,\epsilon)$. By uniform convergence of $u_j$ in $B(p,\epsilon/2)$ we have that $u_{\infty}(X) = 0$ for $X\in B(p,\epsilon/2)$, which implies $B(p,\epsilon/2) \subset \{u_{\infty}=0\}$ and $p\notin \partial\Oinf$. Hence $\partial\Oinf \subset\Lambda_{\infty}$ and we conclude $\Lambda_{\infty} = \partial\Oinf$.		
			

			We now show that $\Gamma_\infty=\overline\Omega_\infty$. Note that if $Z\in\Omega_\infty$, $u_\infty(Z)>0$ and for $j$ large enough $u_j(Z)>0$ also. Thus $Z\in\Omega_j$ 
			for all $j$ large enough and $Z\in \Gamma_\infty$, which yields $\Omega_\infty\subset \Gamma_\infty$. Since $\Gamma_{\infty}$ is closed we conclude $\overline\Omega_\infty\subset \Gamma_\infty$. Let $X\in \Gamma_\infty$. Assume there is $\epsilon>0$ such that $\overline{B(X,2\epsilon)}\subset \Omega_\infty^c$, in particular $B(X,2\epsilon) \cap \partial\Oinf = \emptyset$. Since $\partial\Oinf$ is the limit of $\pOj$'s, for $j$ large enough $B(X,\epsilon) \cap \pOj = \emptyset$. By the definition of $\Gamma_{\infty}$, there is a sequence $X_j\in \overline\Oj $ converging to $X$. Thus for $j$ large enough $ B(X,\epsilon)$ is a neighborhood of $X_j$ and moreover $B(X,\epsilon) \cap\Omega_j \neq \emptyset$ since $X_j\in \overline{\Omega_j}$. Since $B(X,\epsilon) \cap \pOj = \emptyset$
			we conclude that $B(X,\epsilon)\subset \Omega_j$. Using a similar argument to the one used to obtain \eqref{eqn:3.5} we have
			\[
				u_j(X) \geq C(|X|,\epsilon) u_j(A_j(0,1)) \geq C'>0
			\]
			independent of $j$. Hence $u_{\infty}(X) = \lim u_j(X) \geq C'>0$ and $X\in\Oinf$, contradicting the assumption that $X\in \text{int}\,\Omega_\infty^c$. Therefore $X\in\overline\Omega_\infty$, that is $\Gamma_\infty\subset \overline\Omega_\infty$, which concludes the proof of \emph{(2)}.
\end{myproof}
			
\begin{myproof}[Proof of (\ref{im:uniform}) in Theorem \ref{thm:pseudo-blow-geo}]		
		Recall that since $\Omega$ is a uniform domain there is $M>1$ such that for all $q\in \partial\Omega$ and $r\in (0,\diam \Omega)$ there is a point $A(q,r)\in \Omega$ such that
		\begin{equation}\label{eqn106}
B\left(A, \frac{r}{M} \right) \subset B(q,r)\cap\Omega.
	\end{equation} 
Note that since each $\Omega_j$ is a dilation and translation of $\Omega$ \eqref{eqn106} also holds for $q'\in\pOj$ and $r \in (0,\diam \Omega_j)$. 	
			 
		Let $p\in\pOinf$ and $r>0$. Since $\pOj\to\pOinf$, we can find $p_j\in \pOj$ such that $p_j\to p$. For each $j$ there exists $ A_j=A_j(p_j,r/2)$ such that 
	\[
		B\left(A_j, \frac{r}{2M}\right) \subset B\left(p_j, \frac{r}{2}\right) \cap \Oj.
	\]
Note that for $j$ large enough 	
\[
B\left(A_j, \frac{r}{2M}\right) \subset B\left( p_j, \frac{r}{2} \right) \subset \overline{B\left(p,\frac{3r}{4}\right)}. 
\]
Modulo passing to a subsequence (which we relabel)  we can find a point $A(p,r)$ such that $A_j\to A(p,r)$ and for $j$ large enough
\begin{equation} \label{eqn109}
B\left(A(p,r), \frac{r}{3M} \right)\subset B\left(A_j, \frac{r}{2M}\right)  \subset B(p,r) \cap \Oj.
 \end{equation}
 Let $Y\in B\left(A(p,r), \frac{r}{3M} \right)$. By \eqref{eqn109} $u_j(Y)\sim u_j(A_j)$. Since each $\Omega_j$ satisfies the Harnack chain property with the same constant as $\Omega$
 we have that $u_j(A_j)\sim u_j(A_j(0,1))$ with a comparison constant that only depends on $r$ and $|p_j|$ thus for $j$ large enough with a comparison constant that only depends on $r$ and $|p|$. Since
$ u_j(A_j(0,1))\sim 1$, we conclude that $u_j(Y)\sim 1$ with a comparison constant that only depends on $r$ and $|p|$. Hence $u_\infty(Y)>0$ and 
\begin{equation}\label{eq:NTOinf}
		B\left(A(p,r), \frac{r}{3M} \right) \subset B(p,r)\cap\Oinf,  
	\end{equation}
	 		which ensures that $\Omega_\infty$ satisfies the corkscrew condition.

	Fix $X,Y\in \Oinf$. 
	Since $\pOj \to \pOinf$  and $\overline\Oj\to\overline\Omega_\infty$, for $j$ large enough
	\begin{equation}\label{eq:temp109}
	    D[\pOj, \pOinf]\text{ and } D[\overline\Oj, \overline\Oinf] \leq  \frac{d}{2} \min\left\{ \delta_{\infty}(X), \delta_{\infty}(Y) \right\},
	\end{equation}
	here $d \leq 1$ is a constant dependeing on $X, Y$ to be determined later.
	Fix an $j$ sufficiently large, we have $X,Y\in \Oj$ and 
	\[
	 \frac{\delta_{\infty}(X)}{2} \leq \delta_j(X) \leq \frac{3\delta_{\infty}(X)}{2},\quad \frac{\delta_{\infty}(Y)}{2} \leq \delta_j(Y) \leq \frac{3\delta_{\infty}(Y)}{2}. 
	 \]
	Since $\Oj$ satisfies the Harnack chain property with the same constants as $\Omega$, there are constants $c_1<c_2<1$ (independent of $j$) and balls $B_1,\cdots, B_K$ (the choice of balls are dependent of $j$) connecting $X$ to $Y$ in $\Oj$ and such that
	\begin{equation} \label{eq:HBj}
		c_1 \delta_j(B_k) \leq \diam B_k \leq c_2 \delta_j(B_k),  
	\end{equation} 
	for $k=1,2,\cdots,K$ and
	\begin{equation}\label{eqn111}
	 K \leq C\left( \frac{|X-Y|}{\min \{ \delta_j(X), \delta_j(Y) \}} \right) \leq C' \left( \frac{|X-Y|}{\min\{ \delta_{\infty}(X), \delta_{\infty}(Y)\}} \right).
	  \end{equation}
	  Combining \eqref{eq:HBj} and \eqref{eqn111}, we know
	  \begin{equation}\label{eq:temp111}
	      \diam B_k \geq c(c_1,c_2,K) \min\{ \delta_j(X), \delta_j(Y) \} \geq c'(c_1,c_2,K) \min\{ \delta_{\infty}(X), \delta_{\infty}(Y) \} .
	  \end{equation}
	  Combining \eqref{eq:temp111} again with \eqref{eq:HBj}, we find a constant $d=d(c_1,c_2, K)\leq 1$ such that 
	  \[
	      \delta_j(B_k) \geq d \min\{\delta_{\infty}(X),\delta_{\infty}(Y) \} \quad \text{ for all } k = 1,2,\cdots,K.
	  \]
	  Recall \eqref{eq:temp109}, we conclude
	  \[
	      |\delta_j(B_k) - \delta_{\infty}(B_k)| \leq D[\pOj,\pOinf] \leq \frac{d}{2} \min\{\delta_{\infty}(X),\delta_{\infty}(Y) \} \leq \frac{\delta_j(B_k)}{2}.
	  \]
	  Thus $B_k \subset \Oinf$, and moreover,
	  \[
	      \frac{2 \delta_{\infty}(B_k)}{3} \leq \delta_j(B_k) \leq 2 \delta_{\infty}(B_k).
	  \]
	  Combined with \eqref{eq:HBj} we get
	  \begin{equation} \label{eq:HBj-infty}
		    \frac{2}{3} c_1 \delta_\infty(B_k) \leq \diam B_k \leq 2c_2 \delta_\infty(B_k). 
	\end{equation}
    To summarize, we find balls $B_1, \cdots, B_K$ in $\Oinf$ that satisfy \eqref{eq:HBj-infty} and connect $X$ to $Y$, and the number of balls satisfies \eqref{eqn111}. Therefore $\Omega_\infty$ satisfies the Harnack chain condition. This combined with \eqref{eq:NTOinf} shows that $\Oinf$ is a uniform domain with constants comparable to those of $\Omega$.
\end{myproof}
		 
\begin{myproof}[Proof of (\ref{im:hmcv}) and (\ref{im:smcv}) in Theorem \ref{thm:pseudo-blow-geo}]
	As noted right after \eqref{eqn:blow-double} and \eqref{eqn:blow-sigma}, $\{\sj\}$ and $\{\oj\}$ satisfy conditions $i)$ and $ii)$ of Lemma \ref{lm:sptcv}. Moreover for $R>0$,
\[ \sup_j \sigma_j(B(0,R)) = \sup_j \frac{\sigma(B(q_j,Rr_j))}{r_j^{n-1}} \leq CR^{n-1} \]
since $\sigma$ is Ahlfors regular; and 
\[ \sup_j \omega_j(B(0,R)) = \sup_j \frac{\omega(B(q_j,Rr_j))}{\omega(B(q_j,r_j))} \leq C(R) \] 
since $\omega$ is doubling. Therefore modulo passing to a subsequence (which we relabel)
  we have
\[
 \sj \rightharpoonup \sinf, \quad \oj \rightharpoonup \oinf. 
\]
where $\mu_\infty$ and $\oinf$ are Radon measures.
By Lemma \ref{lm:sptcv}, $\mu_\infty$ and $\omega_\infty$ are doubling measures and
\begin{equation} \label{eqn113}
\spt\sj \to \spt\sinf, \quad \spt\oj \to \spt\oinf. 
\end{equation}
Since $\sj$ is Ahlfors regular, it is clear that $\spt\sj = \pOj$; by the doubling property of $\oj$ and that $\oj(B(0,1)) = 1$ we also know $\spt\oj = \pOj$. Recall $\pOj \to \pOinf $, \eqref{eqn113} yields
\[
 \spt\sinf = \spt\oinf = \pOinf. 
\]	
 To show that $\sinf$ is Ahlfors regular let $q\in\partial\Omega_\infty$ and let $q_j\in\pOj$ such that $q_j\to q$. For $r>0$ and $j$ sufficiently large
 	\begin{eqnarray*} \label{eq:lowerAR}
		\mu_\infty(B(q,r)) &\geq &\mu_\infty\left(\overline{B\left(q,\frac{r}{2}\right)}\right)  \geq \limsup \sj\left(\overline{B\left(q,\frac{r}{2}\right)}\right) \nonumber \\
		& \geq& \limsup \sj\left(B\left(q_j,\frac{r}{4}\right) \right) 
		 \geq Cr^{n-1}; 
	\end{eqnarray*} 
	and
	\begin{equation}\label{eq:upperAR}
		\mu_\infty(B(q,r)) \leq \liminf \sj(B(q,r)) \leq \liminf\sj(B(q_j,2r)) \leq C'r^{n-1}.
	\end{equation}
	Note that \eqref{eq:lowerAR} and \eqref{eq:upperAR} guarantee that $\minf$ is Ahlfors regular. Moreover by Theorem 6.9 of \cite{Ma}, there are constants $C_1$ and $C_2$ such that
	\[
	C_1 \minf \leq \mathcal{H}^{n-1}\res{\pOinf} \leq C_2 \minf. 
	\]
\end{myproof}

\begin{myproof}[Proof of Theorem \ref{thm:pseudo-blow-ana}]		
	Let $\psi\in C^\infty_c(\R^n)$.
	Suppose $\spt\psi \subset B(0,R)$ for some large $R>0$. Let $j$ be large enough, so that the pole $X_0 \notin B(q_j,4r_jR)$. Define
	$ \varphi_j(Z) = \psi\left( \frac{Z-q_j}{r_j}\right) $ and note that $\spt \varphi_j\subset B(q_j, r_jR)$ thus $X_0\notin \spt\varphi_j$.
	Using Proposition \ref{representation} as well as a change of variables we have
	\begin{eqnarray*}\label{eqn116}
	 -\int_{\Omega_j }A(q_j+r_jZ)\nabla u_j\cdot\nabla \psi\, dZ	 &=&
	 -\int_\Omega A(X)\frac{r_j^{n-2}} {\omega(B(q_j, r_j))}r_j\nabla G(X_0,X) \cdot r_j\nabla \varphi_j r_j^{-n} dX \nonumber\\
	 &=&  -\frac{1} {\omega(B(q_j, r_j))}\int_\Omega A(X) \nabla G(X_0,X) \cdot \nabla \varphi_j \, dX\nonumber\\
	&=&\frac{1} {\omega(B(q_j, r_j))} \int_{\pO} \varphi_j d\omega^{X_0}\nonumber\\
	&=& \int_{\pOj}\psi\ d\omega_j.
	 \end{eqnarray*}
	  For $Z\in B(0,R)$, we have $	|q-(q_j+r_j Z)| \leq |q-q_j|+r_j R $, since $q_j\to q$, $r_j\to 0$ then
	  $\lim_{j\to\infty} (q_j +r_j Z) = q$. Therefore since  $A\in C(\overline\Omega)$, we have $A(q_j+r_jZ)\to A(q_\infty)$ uniformly on $B(0,R)$.
	 By (1), (2) and (4) in Theorem \ref{thm:pseudo-blow-geo}, $\nabla u_j \rightharpoonup \nabla u_{\infty}$ in $L^2_{loc}(\RR^n)$, $\overline\Oj =\overline{\{u_{j} > 0 \}}\to \overline\Oinf = \overline{\{u_{\infty} > 0 \}}$, and $\partial\Oj \to \partial\Oinf $ in the Hausdorff distance sense. Moreover $\oj \rightharpoonup \oinf$ with $\spt\oj \to \spt \oinf = \partial\Oinf$. Hence letting $j\to\infty$ in \eqref{eqn116} we obtain \eqref{limit-int}.	 
\end{myproof}

\begin{myproof}[Proof of Theorem \ref{thm:AW11}]
Recall we proved in Lemma \ref{A-vanishing-osc} that for $\mathcal{H}^{n-1}$ a.e. $q\in\pO$ there exists $A^*(q)$ a symmetric constant-coefficient elliptic matrix such that \eqref{eqn:v-osc} holds.
For such $q$ consider the blow-up given by Theorem \ref{thm:pseudo-blow-geo} where $q_j=q$ for all $j$. As in the proof of Theorem \ref{thm:pseudo-blow-ana} we have for $\psi\in C^\infty_c(\R^n)$ and $j$ large enough,
	\begin{equation}\label{eqn:3.6}
		-\int_{\Oj} A(q+r_j Z) \nabla u_j \cdot\nabla \psi dZ = \int_{\pOj} \psi d\oj.
	\end{equation}
	Note that as in the proof of Theorem \ref{thm:pseudo-blow-ana}, the right hand side
	\begin{equation}\label{eqn:3.7}
		 \int_{\pOj} \psi d\oj \to  \int_{\pOinf} \psi d\oinf \quad \text{as} \quad j\to\infty.
	\end{equation}
	Recall that for $R>0$, $\sup_{j} \|u_j\|_{W^{1,2}(B(0,R))} \leq C_R <\infty$ by \eqref{eqn:3.4}, $u_j\to u_{\infty}$ in $L^2_{loc}(\RR^n)$ and $\nabla u_j \rightharpoonup \nabla u_{\infty}$ in $L^2_{loc}(\RR^n)$. Thus for $j$ large enough since $\psi \in C_c^{\infty}(\RR^n)$, by triangle inequality and H\"older inequality we have
	\begin{align}
		& \left| \int_{\Oj} \langle A(q_j+r_j Z) \nabla u_j, \nabla \psi \rangle dZ - \int_{\Oinf} \langle A^*(q) \nabla u_{\infty}, \nabla \psi \rangle dZ \right|  \nonumber \\
		& \qquad \leq \left| \int_{\Oj} \left\langle \left( A(q_j+r_j Z)-A^*(q) \right) \nabla u_j, \nabla \psi \right\rangle dZ \right| \nonumber \\
		& \qquad \qquad \qquad + \left| \int_{\Oj} \langle A^*(q) \nabla u_j, \nabla \psi \rangle dZ - \int_{\Oinf} \langle A^*(q) \nabla u_{\infty}, \nabla \psi \rangle dZ \right|  \nonumber \\
		& \qquad \leq \|\nabla \psi\|_{L^{\infty}} \left( \int_{\Oj\cap B(0,R)} |A(q+r_j Z) - A^*(q) |^2 dZ \right)^{\frac{1}{2}} \left( \int_{\Oj\cap B(0,R)} |\nabla u_j|^2 dZ \right)^{\frac{1}{2}} \nonumber \\
		& \qquad \qquad \qquad + \left| \int_{\Oj} \langle A^*(q) \nabla u_{j}, \nabla \psi \rangle dZ - \int_{\Oinf} \langle A^*(q) \nabla u_{\infty}, \nabla \psi \rangle dZ \right|. \label{eqn:3.8}
	\end{align}
	Note that since $A^*(q)$ is a constant-coefficient matrix, $\nabla u_j \rightharpoonup \nabla u_{\infty}$ in $L^2_{loc}(\RR^n)$ implies $A^*(q) \nabla u_j \rightharpoonup A^*(q) \nabla u_{\infty}$ in $L^2_{loc}(\RR^n)$. Thus since $\overline\Oj=\overline{\{u_j > 0 \}} \to \overline\Oinf=\overline{\{u_{\infty} > 0 \}}$
	\begin{equation}\label{eqn:3.9}
		\lim_{j\to\infty} \int_{\Oj} \langle A^*(q) \nabla u_j, \nabla \psi\rangle dZ = \int_{\Oinf} \langle A^*(q) \nabla u_{\infty} dZ, \nabla \psi \rangle dZ.
	\end{equation}
	On the other hand since $A\in L^{\infty}(\Omega)$, and by construction $|A^*(q)| \leq C\|A\|_{L^{\infty}(\Omega)}$, we have that 
	\begin{align}
		\left( \int_{\Oj\cap B(0,R)} |A(q+r_j Z) - A^*(q) |^2 dZ \right)^{\frac{1}{2}} & \leq \left(\frac{1}{r_j^n} \int_{\Omega \cap B(q,r_j R)} |A - A^*(q) |^2 dX \right)^{\frac{1}{2}} \nonumber \\
		& \leq C\|A\|_{L^{\infty}(\Omega)} ^{2-\frac{n}{n-1}} \left( \fint_{\Omega \cap B(q,r_j R)} |A-A^*(q)|^{\frac{n}{n-1}} dX \right)^{\frac{1}{2}}. \label{eqn:3.10}
	\end{align}
	Hence by combining \eqref{eqn:v-osc}, \eqref{eqn:3.8}, \eqref{eqn:3.9} and \eqref{eqn:3.10} we obtain
	\begin{align}
		& \lim_{j\to\infty} \left| \int_{\Oj} \langle A(q_j+r_j Z) \nabla u_j, \nabla \psi \rangle dZ - \int_{\Oinf} \langle A^*(q) \nabla u_{\infty}, \nabla \psi \rangle dZ \right| \nonumber \\
		& \qquad \leq C \|\nabla \psi \|_{L^{\infty}} \sup_j \|\nabla u_j\|_{L^2(B(0,R))} \|A\|_{L^{\infty}(\Omega)} ^{2-\frac{n}{n-1}} \cdot \lim_{j\to\infty} \left( \fint_{\Omega \cap B(q,r_j R)} |A-A^*(q)|^{\frac{n}{n-1}} dX \right)^{\frac{1}{2}} \nonumber \\
		& \qquad \qquad + \lim_{j\to\infty} \left| \int_{\Oj} \langle A^*(q) \nabla u_{j}, \nabla \psi \rangle dZ - \int_{\Oinf} \langle A^*(q) \nabla u_{\infty}, \nabla \psi \rangle dZ \right| =0. \label{eqn:3.11} 
	\end{align}
Thus combining \eqref{eqn:3.7} and \eqref{eqn:3.11}, we conclude the proof of \eqref{limit-int2} and Theorem \ref{thm:AW11}.
\end{myproof}

\section{Analytic properties of the blow-up and pseudo blow-up domains}
As mentioned in section \ref{sect:blowup}, in order to apply Theorem \ref{thm:hmu} we need to study the elliptic measures of the blow-up domain with finite poles. In this section we construct these measures by a limiting procedure which is compatible with the blow-up procedure used to produce the tangent and pseudo-tangent domains.

\begin{theorem}\label{thm:blow-ana-pole}
	Let $\Omega\subset\RR^n$ be a uniform domain with Ahlfors regular boundary. Let $L=-\divg(A\nabla)$ with $A\in C(\overline\Omega)$ or $A\in W^{1,1}(\Omega)\cap L^{\infty}(\Omega)$. 
	Suppose that the elliptic measure $\omega \in A_{\infty}(\sigma)$ in the sense of \cite{HMU}.
	Assume that $\Oinf$ is either the pseudo-tangent or the tangent domain obtained in Theorem \ref{thm:pseudo-blow-geo}, where in the case of $A\in W^{1,1}(\Omega) \cap L^{\infty}(\Omega)$ we use $q_j = q$ for every $j$ and only consider points $q$ satisfying \eqref{eqn:v-osc} and $L_{\infty}$ is the corresponding operator as in Theorem \ref{thm:pseudo-blow-ana} or \ref{thm:AW11}, then $\omega_{L_{\infty}} = \omega_{\infty} \in A_{\infty}(\sigma_\infty)$ in the sense of \cite{HMU} (see Definition \ref{def:AinftyHMU}).
\end{theorem}

\begin{proof}
	Our goal is to show that the elliptic measure of $L_{\infty}$ with finite pole can be recovered as a limit of the elliptic measures of $L_j = -\divg(A_j(Z)\nabla)$ with finite pole, where $A_j(Z) = A(q_j + r_j Z)$ in $\Oj$, 
	and that the $A_{\infty}$ property of elliptic measures is preserved when passing to a limit.
	
	Let $f\in C_c(\RR^n) \cap \Lip(\R^n) \cap W^{1,2}(\RR^n)$, and consider the Dirichlet problem
	\begin{equation}\label{eqn:101A}
		\left\{\begin{array}{rcll}
			L_j v_j & = & 0, & \text{in } \Oj \\
			v_j & = & f, & \text{on } \pOj
		\end{array} \right.
	\end{equation}
	then for $Z \in \Oj$
	\begin{equation}\label{eqn:102A}
		v_j(Z) = \int_{\pOj} f(q) d\omega_j^Z(q).
	\end{equation}
	Here $\omega_j^Z$ is the harmonic measure of $L_j$ in $\Oj$ with pole $Z$. 
	By definitions of $\Oj$ and $ L_j$ it is not hard to see 
	\[
		\omega_j^Z(E) = \omega^{q_j + r_j Z}(q_j + r_j E), \qquad \text{for } E\subset \pOj.
	\]
	
	By the maximum principle 
	\[ \sup_{\Oj} |v_j| \leq \|f\|_{L^{\infty}(\pOj)} \leq \|f\|_{L^{\infty}(\RR^n)} . \]
	Since the domains $\Oj$ have Ahlfors regular boundaries with the same constants, DeGiorgi-Nash-Moser and the assumption that the boundary data $f$ is Lipschitz yield that the solutions $\{v_j\}$ are equicontinuous on compact sets of $\RR^n$.
	Thus the sequence $\{v_j\}$ is equicontinuous and uniformly bounded.
	Furthermore using the variational properties of $v_j$ we know that
	\[
		\int_{\Oj} \langle A_j(Z) \nabla v_j, \nabla v_j \rangle \leq \int_{\Oj} \langle A_j(Z) \nabla f, \nabla f \rangle.
	\]
	The uniform ellipticity of $L_j$ yields
	\[
		\lambda \int_{\Oj} |\nabla v_j|^2 \leq \Lambda \int_{\Oj} |\nabla f|^2. 
	\]
	Extending $v_j = f$ on $\Oj^c$ we have that
	\[ \sup_j \|\nabla v_j \|_{L^2(\RR^n)} \leq \left(\frac{\Lambda}{\lambda}\right)^{\frac{1}{2}} \|\nabla f\|_{L^2(\RR^n)}, \quad \text{and } \sup_j \|v_j\|_{L^2(B(0,R))} \leq C_R. \]
	Modulo passing to a subsequence (which we relabel) we have that there is a continuous function $v\in W^{1,2}_{loc}(\RR^n)$ with $\nabla v \in L^2(\RR^n)$ and such that $v_{j} \to u$ uniformly on compact sets of $\RR^n$ and $\nabla v_{j} \rightharpoonup \nabla v$ in $L^2(\RR^n)$. Note that a priori the choice of a subsequence could depend on the boundary data $f$, which will be problematic. We will show later that this is not the case.
	
	We claim that the function $v$ solves the Dirichlet problem
	\begin{equation}\label{D:Oinf}
		\left\{ \begin{array}{ll}
			L_\infty v = 0, & \text{in }\Oinf \\
			v= f, & \text{on } \pOinf.
		\end{array} \right.
	\end{equation} 
	Note that for $p\in \pOinf$ there exist $p_j \in \pOj$ with $p_j \to p$. Using the continuity of $v$ and $f$ at $p$, the uniform convergence of $v_{j}$ to $v$ on compact sets (for example on $\overline{B(p,r)}$) and the fact that $v_{j}=f$ on $\pOj$, we have
	\begin{align}
		|v(p) - f(p)| & \leq |v(p) - v(p_{j})| + |v(p_j) - v_j(p_j)| + |v_j(p_j) - f(p_j)| + |f(p_j) - f(p)| \nonumber \\
		& \leq |v(p) - v(p_j)| + \|v-v_j\|_{L^{\infty}(\overline{B(p,r)})} + 0 + |f(p_j) - f(p)|. \label{eqn:105A}
	\end{align}
	
	Letting $j\to \infty$ in \eqref{eqn:105A} yields $v(p) = f(p)$. Combined with the continuity of $v$ in $\RR^n$ we conclude that $u$ tends to $f$ continuously towards the boundary $\pOinf$.
	Let $\xi \in C_c^{\infty}(\Oinf)$. Since $\Oj$ and $\Oinf$ are open domains satisfying $\overline{\Oj} \to \overline{\Oinf}$ and $\pOj \to \pOinf$ uniformly on compact sets, a standard argument shows that a compact set contained in $\Oinf$ is eventually contained in $\Oj$. Thus $\xi\in C_c^{\infty}(\Oinf)$ implies $\xi \in C_c^{\infty}(\Oj)$ for $j$ sufficiently large. By \eqref{eqn:101A} we have that
	\begin{equation}\label{eqn:106A}
		\int \langle A(q_j+r_j Z) \nabla v_j,\nabla \xi\rangle dZ = 0.
	\end{equation}
	Letting $j\to \infty$ in \eqref{eqn:106A} and proceeding as in the proofs of Theorem \ref{thm:pseudo-blow-ana} and \ref{thm:AW11}, we conclude that 
	\[
		\int \langle A^*(q) \nabla v, \nabla \xi \rangle dZ = 0,
	\]
	where $A^*(q) = A(q)$ when $A\in C(\overline\Omega)$ and $A^*(q)$ is as in Theorem \ref{thm:AW11} in the case when $A\in W^{1,1}(\Omega) \cap L^{\infty}(\Omega)$. Thus in either case we have $L_\infty v=0$ in $\Omega_\infty$. 
	Since the tangent domain $\Oinf$ is unbounded, the solution to the Dirichlet problem \eqref{D:Oinf} may not be unique. It may not even satisfy the maximum principle. We need more work to show the function $v$ we just constructed is indeed a solution we want, in particular it is indeed the solution that gives rise to the elliptic measure of $\Oinf$.
	
	Suppose that $f$ is compactly supported in $B(0,R_0)$. Given $\epsilon>0$
	there is $j_{\epsilon, R_0}\in \NN$ such that for $j\geq j_{\epsilon, R_0}$, the Hausdorff distance between $\partial\Omega_i\cap \overline{B(0,R_0)}$ and $\partial \Omega_\infty\cap \overline{B(0,R_0)}$ is small enough so that any $p_j \in \pOj \cap \overline{B(0,R_0)}$, there is $p\in \pOinf$ such that since $f$ is uniformly continuous on  $\overline{B(0,R)}$, $|f(p) -f(p_j)|<\epsilon$. Hence 
	\begin{equation}\label{eqn:106C}
	\sup_{\pOj} |f| = \sup_{\pOj \cap \overline{B(0,R)}} |f| \leq \sup_{\pOinf \cap \overline{ B(0,2R)}} |f| + \epsilon = \sup_{\pOinf} |f| + \epsilon.  
	\end{equation}
	For $Z\in \overline{\Oinf}$ there exists a sequence $Z_j \in \overline{\Oj}$ such that $Z_j \to Z$ and all lie in $\overline{B(0,M R_0)}$ for $M$ large enough.
	Since $\sup_{\Oj} |v_j| \leq \sup_{\pOj} |f|$, $v_j \to v$ and $\overline{\Oj} \to \overline{\Oinf}$ uniformly on compact sets. 
	For $\epsilon >0$ there is $j'_{\epsilon, R_0, M}\in \NN$ such that for $j\geq j'_{\epsilon, R_0, M}$, using \eqref{eqn:106C} we have
	\begin{equation}\label{eqn:106D}
		|v(Z)| \leq |v(Z) - v(Z_j)| + |v(Z_j) - v_j(Z_j)| + |v_j(Z_j)| \leq 2\epsilon + \sup_{\pOj}|f|\le 3\epsilon+ \sup_{\pOinf} |f| 
	\end{equation}
	Therefore \eqref{eqn:106D} yields $ \sup_{\Oinf} |v(Z)| \leq 3\epsilon + \sup_{\pOinf} |f| $ for all $\epsilon>0$, and thus $\sup_{\Oinf}|v(Z)| \leq \sup_{\pOinf} |f|$. To summarize,
	 for any $f\in C_c(\RR^n) \cap W^{1,2}(\RR^n)$ we construct a continuous function $v$ satisfying
	\begin{equation}\label{eqn:ellunbd}
		\left\{\begin{array}{ll}
			L_{\infty}v=0, & \text{in } \Oinf \\
			v = f, & \text{on } \pOinf 
		\end{array} \right.
	\end{equation}
	and  satisfying the maximum principle $\sup_{\pOinf} |v| \leq \sup_{\pOinf} |f|$.
	
	  We observe that even though the constructions are different, in the case when the boundary value function $f$ is non-negative and 
	 $f\in C_c(\RR^n) \cap \Lip(\R^n) \cap W^{1,2}(\RR^n)$
	 we produce the same bounded solution $v$ as the one constructed in \cite{HM1} for the unbounded domain $\Omega_\infty$ (see page 13 of \cite{HM1} for details). [Note that the construction in \cite{HM1} is for the Laplacian but holds for any constant coefficient operator]. We denote the solution constructed in \cite{HM1} by $u$. Recall that $u = \lim_{R\to\infty} u_R$, where $u_R$ is the solution to $L_{\infty} u_R = 0$ in the bounded domain $\Omega_R = \Oinf \cap B(0,2R)$ with boundary value $f \eta(\cdot/R)$. Here $\eta$ is a smooth cut-off function such that $0\leq \eta\leq 1$, $\eta = 1$ for $|Z| < 1$ and $\spt \eta \subset \{Z\in\RR^n:|Z|<2\}$. Assume $f$ is compactly supported on $B(0,R_0)$. Then for any $R\geq R_0$, by the maximum principle $u_R \leq v$ in $\Omega_R$, thus the limit $u \leq v$ on $\Oinf$. Set $w=v- u$, it is a non-negative solution to $L_{\infty} w=0$ in $\Oinf$ with vanishing boundary value. Fix $Z\in \Oinf$, since $\Oinf$ is a uniform domain with Ahlfors regular boundary, by Lemma \ref{lem:vanishing}
	  for any $Z\in\Omega_\infty$ with $\delta_{\infty}(Z)<\frac{r}{2}$
\begin{equation}\label{eqn:Holderinfty}
		w(Z) \lesssim \left( \frac{\delta_{\infty}(Z)}{r} \right)^{\beta} \sup_{\Oinf} w \leq 2 \left( \frac{\delta_{\infty}(Z)}{r} \right)^{\beta} \sup_{\pOinf} f,    	\end{equation}
	Letting $r\to\infty$ in \eqref{eqn:Holderinfty} we get $w(Z) = 0$. Thus $v\equiv u$ in $\Oinf$.
	 Recall that at this point for $f\in C_c(\RR^n)\cap \Lip(\R^n) \cap  W^{1,2}(\RR^n)$ we are only able to find a subsequence (possibly depending on $f$) converging to a continuous function $v$ that solves \eqref{eqn:ellunbd}. 
	 We claim that in the case when also $f\ge 0$, the entire sequence $v_j$ converges to $v$. In fact given two arbitrary subsequences $\{v_{j_k}\}$ and $\{v_{j'_k}\}$ of $\{v_j\}$, the 
	 argument above shows that either sequence has a further subsequence that converges to a continuous function, denoted by $v_1$ and $v_2$ respectively. Both functions $v_1$ and 
	 $v_2$ satisfy the equation \eqref{eqn:ellunbd} and maximum principle. Once again by the previous argument they are both equal to $u$, thus $v_1= v_2$ in $\Oinf$. Therefore 
	 the entire sequence $\{v_j\}$ converges to a same continuous function $v= u$. In general if $f$ is not necessarily non-negative, we just decompose it into two non-negative functions $f=f^+-f^-$, with $f^\pm=\max
	 \{0,\pm f \}\ge 0$ and $f^\pm\in C_c(\RR^n)\cap \Lip(\R^n) \cap W^{1,2}(\RR^n)$. The argument above yields $v^\pm$ satisfying  \eqref{eqn:ellunbd} with boundary data $f^\pm$ respectively. Then 
	 $v=v^+-v^-$ is a solution to \eqref{eqn:ellunbd} with boundary data $f$ and satisfying the maximum principle.
	Hence for any $Z\in \Oinf$ fixed, the operator $\Lambda_Z: C_c(\RR^n)\cap\Lip(\R^n)\cap W^{1,2}(\RR^n) \to \RR$ defined by $\Lambda_Z(f)=v(Z)$ is positive bounded 
	(with respect to the $\|\cdot\|_{L^\infty({\R^n})}$ norm) linear functional. Hahn-Banach theorem allows us to extend $\Lambda_Z$ to a positive bounded linear functional on all of $C_c(\R^n)$, with the same operator norm. We still denote the functional as $\Lambda_Z: C_c(\R^n) \to W^{1,2}(\R^n)$.
	By Riesz representation theorem there exists a unique family of Radon measures $\{\omega_{\infty}^Z\}_{Z\in \Oinf}$ such that 
	\[
		\Lambda_Z(f) = \int_{\pOinf} f(q) d\omega_{\infty}^Z(q) \quad \text{ for all } f\in C_c(\R^n).
	\]
	In particular for $f\in C_c(\R^n) \cap \Lip(\R^n) \cap W^{1,2}(\R^n)$, the measures $\{\omega_\infty^Z\}_{Z\in\Oinf}$ satisfies
	\begin{equation}\label{eqn:109A}
		v(Z) = \Lambda_Z(f) = \int_{\pOinf} f(q) d\omega_{\infty}^Z(q).
	\end{equation} 

	
	Recall that the sequence $\{u_j\}$ converges uniformly to $u$ in compact sets. Thus combining \eqref{eqn:102A} and \eqref{eqn:109A} we have that for all $f\in C_c(\RR^n) \cap \Lip(\R^n) \cap W^{1,2}(\RR^n)$
	\begin{equation}\label{eqn:110A}
		\lim_{j\to\infty} \int_{\pOj} f(q) d\omega_j^Z(q) = \int_{\pOinf} f(q) d\omega_{\infty}^Z(q).
	\end{equation}
	A standard approximation argument shows that \eqref{eqn:110A} holds for all $f\in C_c(\RR^n)$.
	And we conclude that $\omega_j^Z \rightharpoonup \omega_\infty^Z$ as Radon measures, for any $Z\in\Oinf$.
	\vspace{.4cm}
	
	To show that $\omega_{\infty}\in A_{\infty}(\sigma_{\infty})$ (recall $\sigma_{\infty} = \mathcal{H}^{n-1}\res{\pOinf}$ is the surface measure) in the sense of $\cite{HMU}$, let $p\in \pOinf$ and $r>0$, and $\Delta' = B(m,s)\cap \pOinf \subset \Delta = B(p,r) \cap \pOinf$ with $m\in\pOinf$. Recall that we denote by $A(p,r)$ a non-tangential point in $\Oinf$ to $p$ at radius $r$; see the proof of Theorem \ref{thm:pseudo-blow-geo} (3). Since $\pOj \to \pOinf$, there exist $p_j \in \pOj$ such that $p_j \to p$ and thus for $j$ large enough $A(p,r) $ is also a non-tangential point to $p_j$ in $\Oj$ with radius $2r$. 
	Since $m\in\pOinf$, there also exist $m_j\in\pOj$ such that $m_j\to m$. In particular for $j$ sufficiently large 
	\begin{equation}\label{eq:mjm}
		|m_j - m| < \frac{s}{5}.
	\end{equation}
	
	 Since the $\Oj$'s are uniform and satisfy the CDC with the same constant and since the operators $L_j$'s have ellipticity constants bounded by $\lambda$ and $\Lambda$, we conclude from Corollary \ref{ellip-lb} and Lemma \ref{doubling} that $\oj^{A(p,r)}$ is doubling with a universal constant (independent of $j, p, r$), denoted by $C$.
		Hence by Theorem 1.24 in \cite{Ma} we have
		\begin{align}
			\oinf^{A(p,r)} (\Delta(m,s)) \geq \oinf^{A(p,r)} \left( \overline{\Delta\left(m,\frac{4}{5}s\right)} \right) & \geq \limsup_{j\to\infty} \oj^{A(p,r)} \left( \overline{\Delta\left(m,\frac{4}{5}s\right)} \right) \nonumber \\
			& \geq \limsup_{j\to\infty} \oj^{A(p,r)} \left( \overline{\Delta\left(m_j,\frac{3}{5}s\right)} \right) \nonumber \\
			& \geq C^{-1}  \limsup_{j\to\infty} \oj^{A(p,r)} \left( \Delta\left(m_j,\frac{6}{5}s\right) \right). \label{eqn:113A}
		\end{align}
		Let $V$ be an arbitrary open set in $B(m,s)$, by \eqref{eq:mjm}
		\[
			V\subset B(m,s) \subset B\left(m_j, \frac{6}{5} s\right).
		\] 
		Again by Theorem 1.24 in \cite{Ma} and \eqref{eqn:113A} we have
		\begin{align}
			\frac{\oinf^{A(p,r)}(V)}{\oinf^{A(p,r)}(\Delta(m,s))} & \leq C \dfrac{\liminf_{j\to\infty} \ojA(V)}{\limsup_{j\to\infty} \oj^{A(p,r)} \left( \Delta\left(m_j,\frac{6}{5}s\right) \right) } \nonumber \\
			 & \leq C \liminf_{j\to\infty} \left( \dfrac{\ojA(V) }{\oj^{A(p,r)} \left( \Delta\left(m_j,\frac{6}{5}s\right) \right) }\right). \label{eqn:114A}
		\end{align}
	Let $\wRj = q_j + r_j m_j$ and $\widetilde p_j = q_j + r_j p_j$ in $\pO$, by the definition \eqref{eqn:blow3} of $\oj$,
	\begin{equation}\label{eq:rescaleoj}
		\dfrac{\ojA(V) }{\oj^{A(p,r)} \left( \Delta\left(m_j,\frac{6}{5}s\right) \right) } = \dfrac{\omega^{q_j + r_j A(p,r)}(q_j + r_j V) }{\omega^{q_j + r_j A(p,r)}\left( \Delta \left(\wRj,\frac{6}{5}sr_j\right)\right) }.
	\end{equation}
	The assumption $B(m,s) \subset B(p,r)$ implies $|m-p| \leq r-s$. Thus by $m_j\to m$, $p_j\to p$ we have
	\[
		|m_j - p_j| \leq |m_j - m| + |m-p| + |p-p_j| < r-\frac{s}{5}.
	\]
	Note $s<r$, hence
	\[
		\Delta\left( m_j, \frac{6}{5} s\right) \subset \Delta(p_j,2r).
	\]
	Recall that $A(p,r)$ is a non-tangential point in $\Omega_j$ to the boundary point $p_j$ at radius $2r$. Therefore after rescaling from $ \Omega_j$ to $\Omega$, we have that $q_j + r_j A(p,r)$ is a non-tangential point to the boundary point $\widetilde p_j$ at radius $2rr_j$, and that
	\[
		q_j + r_j V \subset \Delta \left(\wRj, \frac{6}{5} s r_j \right) \subset \Delta(\widetilde p_j, 2rr_j).
	\]
	By the assumption that $\omega_L \in A_\infty(\sigma)$ (see Definition \ref{def:AinftyHMU}), we conclude that 
	\begin{equation}
		\dfrac{\omega^{q_j + r_j A(p,r)}(q_j + r_j V) }{\omega^{q_j + r_j A(p,r)}\left( \Delta \left(\wRj,\frac{6}{5}sr_j\right)\right) } \leq C \left( \dfrac{\mathcal{H}^{n-1}\left(\pO\cap \left(q_j + r_j V\right)\right) }{ \mathcal{H}^{n-1} \left( \pO \cap B\left( \wRj, \frac{6}{5} sr_j \right) \right) } \right)^{\theta}. \label{eq:Ainftypole}
	\end{equation}
	Combining \eqref{eqn:114A}, \eqref{eq:rescaleoj} and \eqref{eq:Ainftypole}, using the definition \eqref{eqn:blow3} of $\sj$, $\sj\rightharpoonup \sinf$ and that $\sigma = \mathcal{H}^{n-1}|_{\pO}$ and $\sinf$ are Ahlfors regular with the the same constant, we get
		\begin{align}
			\dfrac{\omega_{\infty}^{A(p,r)}(V) }{\omega_{\infty}^{A(p,r)}(\Delta(m,s)) } & \leq 
			C \liminf_{j\to\infty} \left( \dfrac{\mathcal{H}^{n-1}\left(\pO\cap \left(q_j + r_j V\right)\right) }{ \mathcal{H}^{n-1} \left( \Delta\left( \wRj, \frac{6}{5} s r_j \right) \right) } \right)^{\theta} \nonumber \\
			& \lesssim \left( \liminf_{j\to\infty} \frac{\sigma(q_j + r_j V)}{(r_j s)^{n-1}} \right)^{\theta} \nonumber \\
		& \leq \left( \frac{1}{s^{n-1}} \liminf_{j\to\infty} \sj(V)  \right)^{\theta} \nonumber \\ 
		& \leq \left( \frac{\sinf(\overline V)}{\sinf(\Delta(m,s))}   \right)^{\theta}, \label{eqn:117A}
		\end{align}
		Recall that $\mu_{\infty}$ is equivalent to the surface measure $\sigma_{\infty} = \mathcal{H}^{n-1} \res{\pOinf}$ (see  Theorem \ref{thm:pseudo-blow-geo} \eqref{im:smcv}). Hence \eqref{eqn:117A} yields that for any open set $V\subset \Delta(m,s) \subset \Delta(p,r)$ with $p, m \in \pOinf$
		\begin{equation}\label{eqn:118A}
			\dfrac{\omega_{\infty}^{A(p,r)}(V) }{\omega_{\infty}^{A(p,r)}(\Delta(m,s)) } \leq C \left( \frac{\sigma_{\infty}(\overline V)}{\sigma_{\infty}(\Delta(m,s))} \right)^\theta.
		\end{equation} 
		
		For any $E\subset B(m,s)$ closed, since $\sigma_{\infty}$ is a Radon measure, given any $\epsilon>0$ there is an open set $V$ satisfying $E\subset V \subset B(m,s)$ and $\sigma_{\infty}(V\setminus E) <\epsilon$. Note that for any $x\in E$, there is $r_x >0 $ such that $B(x,2r_x) \subset V$ and $E\subset \cup_{x\in E} B(x,r_x)$. Since $E$ is compact we can extract a finite subcover $E\subset \cup_{i=1}^m B(x_i, r_i) =  U$ and $B(x_i, 2r_i) \subset V$ for $i= 1,\cdots,m$. Note that $E\subset U \subset \overline U \subset V$. Thus $\sigma_{\infty}(\overline U \setminus E) <\epsilon$, and using \eqref{eqn:118A} we have
		\begin{align*}
			\dfrac{\oinf^{A(p,r)}(E) }{\oinf^{A(p,r)}(\Delta(m,s)) } \leq \dfrac{\oinf^{A(p,r)}(U) }{\oinf^{A(p,r)}(\Delta(m,s)) } \leq C \left( \frac{\sigma_{\infty}(\overline U)}{\sigma_{\infty}(\Delta(m,s))} \right)^\theta \leq C \left( \frac{\sigma_{\infty}(E)+\epsilon}{\sigma_{\infty}(\Delta(m,s))} \right)^\theta. \label{eqn:119A}
		\end{align*}
		Letting $\epsilon \to 0$ we have that for any closed set $E\subset B(m,s)$
		\begin{equation}\label{eqn:120A}
			\dfrac{\oinf^{A(p,r)}(E) }{\oinf^{A(p,r)}(\Delta(m,s)) } \leq C \left( \frac{\sigma_{\infty}(E)}{\sigma_{\infty}(\Delta(m,s))} \right)^\theta. 
		\end{equation}
		Since both $\oinf^{A(p,r)}$ and $\sigma_{\infty}$ are Radon measures, \eqref{eqn:120A} holds for any Borel set $E\subset B(m,s)$, which concludes the proof of Theorem \ref{thm:blow-ana-pole}.

\end{proof}

\begin{corollary}\label{cor:blow-up-nta}
	Let $\Omega\subset \RR^n$ be a uniform domain with Ahlfors regular boundary. Let $L=-\divg(A\nabla)$ with $A\in C(\overline\Omega)$ (resp. $A\in W^{1,1}(\Omega) \cap L^{\infty}(\Omega)$). 
	Suppose that the elliptic measure $\omega \in A_{\infty}(\sigma)$ in the sense of \cite{HMU}.
	Then any pseudo-tangent domain $\Oinf$ (resp. tangent domain at a point $q\in \pO$ satisfying \eqref{eqn:v-osc}) is an NTA domain with constants depending only on the allowable constants.
\end{corollary}

\begin{proof}
	Theorem \ref{thm:blow-ana-pole} combined with Theorem \ref{thm:hmu} ensures that under the hypotheses of Theorem \ref{thm:lnta} (resp. Theorem \ref{thm:lfp}), all pseudo blow-ups of $\Omega$ (resp. all blow-ups of $\Omega$ at points $q\in \pO$ satisfying \eqref{eqn:v-osc}) are uniform domains with uniformly rectifiable boundaries with constants depending on the allowable constants. By \cite{AHMNT} we conclude that all such domains are NTA domains with constants depending only on the allowable constants.
\end{proof}

\section{Proof of Theorems \ref{thm:lfp} and \ref{thm:lnta}}

Given Corollary \ref{cor:blow-up-nta}, we may assume that all pseudo-tangent domains in the case $A\in C(\overline\Omega)$ or tangent domains at points $q\in\pO$ satisfying \eqref{eqn:v-osc} in the case $A\in W^{1,1}(\Omega) \cap L^{\infty}(\Omega)$ are NTA domains with exterior corkscrew constant $M_{\infty}$. That is if $\Oinf$ is obtained via this blow-up procedure then for any $p\in \pOinf$ and $r>0$, there exists $A_{\infty}^-(p,r)\subset \Oinf^c \cap B(p,r)$ such that
\begin{equation}\label{eqn:201A}
	B\left(A_{\infty}^-(p,r), \frac{r}{M_{\infty}} \right) \subset \Oinf^c \cap B(p,r)
\end{equation}
and in particular
\[
	\dfrac{\mathcal{H}^n(\Oinf^c \cap B(p,r))}{ r^n} \geq \frac{c_n}{M_{\infty}^n} > 0 \quad \text{ for any }r>0.
\]

\begin{myproof}[Proof of Theorem \ref{thm:lnta}]
	We want to show that there exists $r_\Omega>0$, such that $\Omega$ satisfies the exterior corkscrew condition with constant $2M_{\infty}$ for all $q\in \partial\Omega$ and all $r<r_\Omega$. Assume that such an $r_\Omega$ does not exist, then there are sequences $r_j\to 0$ and $q_j\in\pO$ such that we cannot find a corkscrew point in $\Omega^c$ with constant $2 M_{\infty}$ at $q_j\in\pO$ with radius $r_j$. Consider $\Oj = (\Omega- q_j)/r_j$, then apply Theorem \ref{thm:pseudo-blow-geo}, Corollary \ref{cor:blow-up-nta} and \eqref{eqn:201A} to find a point $A^-_{\infty}(0,1) \subset \Oinf^c \cap B(0,1) $ such that $B(A^-_{\infty}(0,1), 1/M_{\infty}) \subset \Oinf^c \cap B(0,1)$. Since $\overline\Oj \to \overline\Oinf$ locally uniformly on compact sets, for $j$ large enough
	\[
		B\left( A^-_{\infty}(0,1), \frac{1}{2M_{\infty}} \right) \subset \Oj^c \cap B(0,1),
	\]
	which implies
	\begin{equation}\label{eqn:205A}
		B\left( A_j, \frac{r_j}{2M_{\infty}} \right) \subset \Omega^c \cap B(q_j,r_j) \quad \text{with } A_j = q_j + r_j A^-_{\infty}(0,1).
	\end{equation}
	This contradicts our assumption.
\end{myproof}

\begin{myproof}[Proof of Theorem \ref{thm:lfp}]
	Let $q\in\pO$ such that \eqref{eqn:v-osc} holds. Recall this occurs for $\mathcal{H}^{n-1}$ a.e. $q\in\pO$ (see Lemma \ref{A-vanishing-osc}). Since $\Omega$ satisfies the interior corkscrew condition (with a constant $M$), for any $q\in\pO$,
	\begin{equation}\label{eqn:206A}
		\liminf_{r\to 0} \dfrac{\mathcal{H}^n(\Omega \cap B(q,r))}{r^n} \geq \frac{c_n}{M} >0.
	\end{equation}
	Let $r_j \to 0$ and $\Oj = (\Omega- q)/r_j$. By Theorem \ref{thm:pseudo-blow-geo}, Corollary \ref{cor:blow-up-nta}, \eqref{eqn:201A} and a similar argument as in \eqref{eqn:205A}, we have that for a subsequence (which we relabel) $B(A_j, r_j/2M_{\infty}) \subset \Omega^c \cap B(q,r_j)$, where $A_j = q+r_j A^-_{\infty}(0,1)$. Thus
	\begin{equation}\label{eqn:207A}
		\liminf_{j\to\infty} \dfrac{\mathcal{H}^n(\Omega^c \cap B(q,r_j))}{r_j^n} \geq \frac{c_n}{(2M_{\infty})^n} >0.
	\end{equation} 
	Combining \eqref{eqn:206A} and \eqref{eqn:207A}, we conclude that such $q$ belongs to the measure-theoretic boundary $\partial_*\Omega$ of $\pO$, thus $\mathcal{H}^{n-1}(\pO \setminus \partial_*\Omega) = 0$. Since $\mathcal{H}^{n-1}\res{\pO}$ is Ahlfors regular, it is in particular locally finite. Theorem 1 of Section 5.11 in \cite{Evans} ensures that $\Omega$ is a set of locally finite perimeter. Thus the reduced boundary $\partial^*\Omega$ is rectifiable. Since $\mathcal{H}^{n-1}(\partial_*\Omega \setminus \partial^*\Omega) = 0$ the fact that 
	$\mathcal{H}^{n-1}(\pO \setminus \partial_*\Omega) = 0$ implies $\mathcal{H}^{n-1}(\pO \setminus \partial^*\Omega) = 0$. We conclude that $\pO$ is rectifiable.
	\end{myproof}

\section{Qualitative case: reduction to local quantitative case}\label{qualitative}

In this section we discuss how the quantitative approach also yields information about the qualitative case.
Theorem \ref{thm:qc} is proved by reducing it to the following situation which can be seen as a local version of Theorem \ref{thm:lfp}.


\begin{theorem}\label{thm:lfplocal}
Let $\Omega\subset \mathbb{R}^n$ be a bounded uniform domain with Ahlfors regular boundary. Let $L = -\divg(A(X)\nabla)$ with $A\in W^{1,1}(\Omega)\cap L^{\infty}(\Omega)$ satisfying \eqref{def:UE}. Suppose that $G\subset \pO$ is an open set. Assume there are uniform positive constants $C_0,\theta$ so that for any surface ball $\Delta= \Delta(q,r) \subset G$, the elliptic measure with pole at $A_{\Delta}$ satisfies
\begin{equation}\label{eq:AinftyD}
	\frac{\omega^{A_\Delta}(E)}{\omega^{A_\Delta}(\Delta')} \leq C_0 \left( \frac{\sigma(E)}{\sigma(\Delta')} \right)^\theta,
\end{equation}
where $A_{\Delta}$ is a non-tangential point with respect to $\Delta$, and \eqref{eq:AinftyD} holds for all surface balls $\Delta'\subset \Delta$ and all Borel sets $E\subset \Delta'$.
Then $G$ is $(n-1)$-rectifiable. 
\end{theorem}

\begin{remark}
	Note that the assumption \eqref{eq:AinftyD} is a local version of $\omega \in A_{\infty}(\sigma) $ in the sense of \cite{HMU}. Recall that the proof of Theorem \ref{thm:lfp} consists of understanding the 
	blow-ups of the domain $\Omega$ at some $q\in\pO$ and showing that the $A_{\infty}$ property of the elliptic measure holds for the tangent domain $\Oinf$. Since tangent objects only provide infinitesimal information at the blow-up point it is not surprising that only local assumptions are necessary to obtain rectifiablilty.

\end{remark}

\begin{proof}
	If $G$ is empty, there is nothing to prove, so we assume $G\neq \emptyset$.
	Since $G$ is an open subset of an Ahlfors regular boundary and $\sigma \ll \omega$, we have $\sigma(G)>0$ and thus $\omega(G)>0$.
	Consider 
	\[
	\widehat G= \big\{q\in G: A^*(q) \text{ exists as in Lemma \ref{A-vanishing-osc} } \big\}, 
	\]
	then $\sigma(G\setminus \widehat G) = 0$. 
	Theorems \ref{thm:pseudo-blow-geo} and \ref{thm:AW11} hold if we consider a geometric blow-up at $q\in \widehat G$. We claim that the tangent domain $\Omega_{\infty}$ at every $q\in\widehat G$ satisfies 
	its elliptic measure $\omega_{\infty} $ is of class $A_\infty$ with respect to the surface measure $\sigma_\infty = \mathcal{H}^{n-1}|_{\pOinf}$ in the sense of \cite{HMU}, i.e. Theorem \ref{thm:blow-ana-pole} holds.
	Namely we need to show for any point $p\in \pOinf$, any surface ball $\Delta(m,s)=B(m,s)\cap\pOinf \subset B(p,r)\cap \pOinf$ with $m\in \pOinf$, $r,s >0$ and any open subset $V$ of $B(m,s)$, 
	\begin{equation}\label{eqn6.1}
	\dfrac{\omega_{\infty}^{A(p,r)}(V) }{\omega_{\infty}^{A(p,r)}(\Delta(m,s)) } \leq C \left( \frac{\sigma_{\infty}(\overline V)}{\sigma_{\infty}(\Delta(m,s))} \right)^\theta.
	\end{equation}
	 Note that in the proof of Theorem \ref{thm:blow-ana-pole}, the construction of the elliptic measure $\omega_{\infty}$ does not require the $A_{\infty}$ property of $\omega$. It only uses the fact that $\Omega$ is a uniform domain with Ahlfors regular boundary and that $A^*(q)$ exists. Moreover, 	  we still have $\omega_j^Z \rightharpoonup \omega_{\infty}^Z$ for any $Z\in \Omega_{\infty}$.  
 	Recall the notations in the proof of Theorem \ref{thm:blow-ana-pole} (note that in this case,  we have the blow-up point $q_j = q$ for all $j$): there are sequences $\pOj \ni p_j \to p\in\pOinf$, $\pOj \ni m_j \to m \in\partial\Omega_{\infty}$, and we let $\widetilde m_j = q+r_j m_j$, $\widetilde p_j = q+ r_j p_j$ on $\pO$. A close look at the proof of Theorem \ref{thm:blow-ana-pole} shows that to prove \eqref{eqn6.1} it is enough to show \eqref{eq:Ainftypole}, which we rewrite here:
	\begin{equation}\label{copyeq:Ainftypole}
		\dfrac{\omega^{q + r_j A(p,r)}(q_j + r_j V) }{\omega^{q + r_j A(p,r)}\left(\Delta \left( \wRj, \frac{6}{5} sr_j \right)\right) } \leq C \left( \dfrac{\mathcal{H}^{n-1}\left(\pO\cap \left(q + r_j V\right)\right) }{ \mathcal{H}^{n-1} \left( \pO \cap B\left( \wRj, \frac{6}{5} sr_j \right) \right) } \right)^{\theta}.
	\end{equation}
	Moreover since $G$ is open, for any $q\in\widehat G\subset G$ we can find a surface ball $\Delta_0 = \Delta(q,r_0) \subset G$. Hence if $j$ is large enough (so that $r_j$ is small enough), we have the surface ball
	\begin{equation}\label{eqn6.100}
	\Delta(\widetilde p_j, 2rr_j) \subset \Delta\left(q, 2(r+|p|) r_j \right) \subset \Delta(q,r_0) 
	\end{equation}
is contained in $G$. Therefore we may apply the assumption \eqref{eq:AinftyD} to the surface ball 
\[ \Delta' = \Delta \left( \wRj, \frac{6}{5} sr_j \right)=B(\widetilde m_j, 6 s r_j/5)\cap\pO\subset \Delta(\widetilde p_j, 2rr_j), \] 
with non-tangential pole $q+r_j A(p,r)$ and to the Borel set $E= q+r_j V\subset \Delta \left( \wRj, \frac{6}{5} sr_j \right)$ and obtain \eqref{copyeq:Ainftypole}. (Recall that $\sigma = \mathcal{H}^{n-1}|_{\pO}$.)
	By the same argument as in the proof of Theorem \ref{thm:blow-ana-pole} we conclude that the tangent domain $\Omega_{\infty}$ satisfies $\omega_{\infty} \in A_{\infty}(\sigma_{\infty})$ in the sense of \cite{HMU} . 
	Hence as in the Theorem \ref{thm:lfp}, we have that $\widehat G \subset \partial_*\Omega$, where $\partial_*\Omega$ is the measure-theoretic boundary of $\Omega$. A local version of Theorem 1 in Section 5.11 in \cite{Evans} ensures that $\widehat G$ is rectifiable, and so is $G$.
\end{proof}

Before reducing Theorem \ref{thm:qc} to Theorem \ref{thm:lfplocal}, we recall some results on uniform domains with the CDC which are needed for the proof.
\begin{lemma}[Change of pole formula]
	Let $\Omega$ be a bounded uniform domain satisfying the CDC and $L=-\divg(A(X)\nabla)$ be an elliptic operator satisfying \eqref{def:UE}. Let $X_0\in\Omega$ be fixed and denote the elliptic measure by $\omega = \omega^{X_0}$.  Suppose $q\in\pO$ and $r<\diam \Omega/4$ are such that $X_0 \notin B(q,4r)$, we denote $\Delta= B(q,r)\cap\pO$. Then for any surface ball $\Delta' \subset \Delta$ we have
	\begin{equation}\label{eq:cop}
		\frac{ \omega(\Delta')}{\omega(\Delta)} \sim \omega^{A_{\Delta}}(\Delta'),
	\end{equation}
	where $A_{\Delta}$ is a non-tangential point to surface ball $\Delta$.
\end{lemma}
\begin{proof}
	By Corollary \ref{ellip-lb}, we know \eqref{eq:cop} follows directly from
	\[ \frac{ \omega(\Delta')}{\omega(\Delta)} \sim \frac{\omega^{A_{\Delta}}(\Delta')}{ \omega^{A_\Delta}(\Delta)}, \] 
	i.e. the boundary comparison principle. See \cite{Zh} for the proof of the comparison principle when $\Omega$ is a uniform domain with Ahlfors regular boundary. For the case when we only assume $\Omega$ satisfies the CDC, the proof is to appear in detail in \cite{HMT2}.
\end{proof}

In fact \eqref{eq:cop} holds if we replace the surface ball $\Delta'$ by any Borel set $E\subset \Delta$, i.e.
\begin{equation}\label{eq:copBorel}
	\frac{ \omega(E)}{\omega(\Delta)} \sim \omega^{A_{\Delta}}(E).
\end{equation}
Suppose $V \subset \Delta$ is (relative) open in $\pO$. For any $x\in V$ let $\Delta_x = \Delta(x,r_x)$ be a surface ball satisfying $\Delta_x \subset V$ with $r_x<\frac{\delta(A_\Delta)}{16}$, then  $ V \subset \cup_{x\in V} \Delta_x$. By Vitali covering lemma we may extract a countable collection of pairwise disjoint balls $\{\Delta_{j}\}_{j\in J}$ such that
\begin{equation}\label{eq:copcoverbycube}
	V \subset \bigcup_{j\in J}  5 \Delta_{j}, \quad\text{where } 5\Delta_{j} := \Delta(x_j,5r_{x_j}).
\end{equation}
 By \eqref{eq:cop}, \eqref{eq:copcoverbycube} and the doubling properties of $\omega$ and $\omega^{A_{\Delta}}$, we have
\begin{align}\label{eqn6.2}
	\frac{\omega(V)}{\omega(\Delta)} \leq \sum_{j\in J} \frac{\omega\left(   5 \Delta_{j} \right)}{\omega(\Delta)} \le C \sum_{j\in J} \frac{\omega\left(   \Delta_{j} \right)}{\omega(\Delta)} \le C \sum_{j\in J} \omega^{A_{\Delta}}(\Delta_{j}) = C\omega^{A_{\Delta}}\left( \bigcup_{j\in J} \Delta_{j} \right) \leq C\omega^{A_{\Delta}}(V),
\end{align}
and similarly
\begin{align}\label{eqn6.3}
	\omega^{A_{\Delta}}(V) \leq \sum_{j\in J} \omega^{A_{\Delta}} (5\Delta_{j}) \le C\sum_{j\in J} \omega^{A_{\Delta}} (\Delta_{j}) \le C\sum_{j\in J} \frac{\omega(\Delta_{j})}{\omega(\Delta)} = C
	\frac{\omega\left( \bigcup_{j\in J} \Delta_{j} \right)}{\omega(\Delta)} \le C \frac{\omega(V)}{\omega(\Delta)}.
\end{align}
Now suppose $E$ is a Borel set contained in $\Delta$. Since $\omega^{A_{\Delta}}$ is a Radon measure, for any $\epsilon>0$ we can find an open set $V_{\epsilon} \supset E$ such that 
$\omega^{A_{\Delta}}(V_{\epsilon}\backslash E) <\epsilon$. We may assume $V_{\epsilon} \subset \Delta$ (if not, just replace $V_{\epsilon}$ by $V_{\epsilon} \cap \Delta $). Combined with \eqref{eqn6.2} we get
\begin{equation}\label{eqn6.4}
	\frac{\omega(E)}{\omega(\Delta)} \leq \frac{\omega(V_{\epsilon} )}{\omega(\Delta)} \le C \omega^{A_{\Delta}}(V_{\epsilon} ) \le C\left( \omega^{A_{\Delta}}(E) + \epsilon\right).
\end{equation}
Passing $\epsilon \to 0$ we get $\omega(E)/\omega(\Delta) \lesssim \omega^{A_{\Delta}}(E)$. By taking a different open set $V'_{\epsilon} \supset E$ satisfying $\omega(V'_{\epsilon}\backslash E)<\epsilon$, we can similarly use \eqref{eqn6.3} to show $\omega^{A_{\Delta}}(E) \lesssim \omega(E)/\omega(\Delta) $. This finishes the proof of \eqref{eq:copBorel}.\qed

\begin{lemma}[Dyadic grids on Ahlfors regular set, see \cite{DSSI}, \cite{DSUR}, \cite{Ch}]\label{lm:dgrid}
	Let $\Omega$ be a domain with Ahlfors regular boundary. There exist positive constants $a_0, \eta$, and $C_1$ depending only on $n$ and the Ahlfors regular constants, such that for each $k\in \mathbb{Z}$ there is a collection of Borel sets (``cubes'')
	\begin{equation*}
		\mathbb{D}_k := \{Q_j^k \subset \pO: j\in \mathcal{J}_k \},
	\end{equation*}
	where $\mathcal{J}_k$ denotes some (possibly finite) index set depending on $k$, satisfying
	\begin{enumerate}[(i)]
		\item $\pO= \cup_j Q_j^k$ for each $k\in \mathbb{Z}$;
		\item if $m\geq k$, then either $Q_i^m \subset Q_j^k$ or $Q_i^m \cap Q_j^k = \emptyset$;
		\item for each $(j,k)$ and each $m<k$, there is a unique $i$ such that $Q_j^k \subset Q_i^m$;
		\item $\diam Q_j^k \leq C_1 2^{-k}$;
		\item each $Q_j^k$ contains some ``surface ball'' $\Delta(x_j^k, a_0 2^{-k}) = B(x_j^k, a_0 2^{-k}) \cap \pO$;
		\item $\mathcal{H}^{n-1} \left( \{x\in Q_j^k: \dist(x,\pO \setminus Q_j^k) \leq \tau 2^{-k} \} \right) \leq C_1 \tau^\eta \mathcal{H}^{n-1}(Q_j^k)$ for all $k,j$ and all $\tau\in (0,a_0)$.
	\end{enumerate}
\end{lemma}

\begin{myproof}[Proof of Theorem \ref{thm:qc}] 
Let $k_0 \in \mathbb{Z}$ be the smallest integer such that $C_1 2^{-k_0} \leq \diam\pO$. We consider a dyadic grid $\mathbb{D} = \{Q\in \mathbb{D}_k: k \geq k_0 \}$ of the Ahlfors regular set $\pO$. Since $\pO$ is bounded, by property $(v)$ of Lemma \ref{lm:dgrid} there are finitely many cubes in the collection $\mathbb{D}_{k_0}$. For each $Q\in \mathbb{D}_{k_0}$ we have $\sigma(Q)\sim \left(2^{-k_0} \right)^{n-1} >0$. Since  $\sigma\ll\omega$ this implies $\omega(Q)>0$. 
Now let $N_0\in\mathbb{N}$ be the smallest integer such that
\[
	\frac{1}{N_0} \leq \min_{Q\in\mathbb{D}_{k_0}} \frac{\omega(Q)}{\sigma(Q)} \leq \max _{Q\in\mathbb{D}_{k_0}} \frac{\omega(Q)}{\sigma(Q)} \leq N_0.
\]
We apply a stopping time argument to the descendants of each cube $Q\in\mathbb{D}_{k_0}$.
Let $N\geq N_0 $ be an integer and let $\mathcal{F}_N = \{B_l\} \subset \mathbb{D} $ be the collection of maximal ``bad'' dyadic cubes with respect to the ``stopping criterion'' that
\[
	\text{either }\qquad\frac{\omega(B_l)}{\sigma(B_l)}< \frac{1}{N}\qquad \text{ or } \qquad\frac{\omega(B_l)}{\sigma(B_l)}> N.
\]
In particular $Q$ is not (a descendent of) a cube in $\mathcal{F}_N$ if it satisfies
\[
\frac{1}{N}\le \frac{\omega(Q)}{\sigma(Q)}\le N.
\]
Let 
\begin{equation}\label{eqn6.6}
 \Lambda_N = \pO \setminus \bigcup_{B_l\in \mathcal{F}_N} B_l. 
 \end{equation}
 Note that $\Lambda_N\subset\Lambda_{N+1}$ and
\begin{equation}\label{eq:pOdecomp}
	\pO = \left( \bigcap_{N\geq N_0} \bigcup_{B_l\in \mathcal{F}_N} B_l\right) \bigcup \left( \bigcup_{N\geq N_0} \Lambda_N \right) 
	=: R_0\bigcup \left(\bigcup_{N\geq N_0} \Lambda_N\right).
\end{equation}
We claim that $\sigma(R_0) = 0$.
In fact by the definition of $R_0$, each $q\in R_0$ is contained in some bad cube $B^{(N)} \in \mathcal{F}_N$, satisfying
for every $N\geq N_0$
\[ \text{either } \qquad\frac{\sigma(B^{(N)})}{\omega(B^{(N)})}  > N \qquad \text{ or } \qquad\frac{\sigma(B^{(N)})}{\omega(B^{(N)})} < \frac{1}{N}. \]
Hence every $q\in R_0$ falls into one of two cases: 
\begin{itemize}
\item there is a sequence $N_i \to \infty$ such that $\sigma(B^{(N_i)})/\omega(B^{(N_i)}) > N_i$ for all $i$, in which case we say $q\in R_0^{b}$ 
\item there is a sequence $N'_i\to\infty$ such that $\sigma(B^{(N'_i)})/\omega(B^{(N'_i)}) < 1/N_i$ for all $i$, in which case we say $q\in R_0^s$. 
\end{itemize}
Note that both $R_0^b$ and $R_0^s$ are Borel sets.
Since $\sigma \ll \omega$, the Radon-Nikodym derivative $h = d\sigma/d\omega$ is in $L^1(\omega)$ and is finite $\omega$-almost everywhere. 
Therefore by the Lebesgue differentiation theorem,
\begin{equation}\label{eq:caseb}
	h(q) = \infty \text{ for } \omega \text{-a.e. } q\in R_0^b, 
\end{equation} 
and
\begin{equation}\label{eq:cases}
	h(q) = 0 \text{ for } \omega \text{-a.e. } q\in R_0^s.
\end{equation}  
Since $h$ is finite $\omega$-almost everywhere, \eqref{eq:caseb} implies that $\omega(R_0^b) = 0$, and thus $\sigma(R_0^b) = 0$. On the other hand by \eqref{eq:cases} we have $\sigma(R_0^s) = \int_{R_0^s} h d\omega = 0$ since $\omega$ is a finite measure.
We conclude that $\sigma(R_0) = \sigma(R_0^b \cup R_0^s) = 0$.
Hence to show $\pO$ is rectifiable, it suffices to show $\Lambda_N $ is rectifiable for all $N\ge N_0$ (see \eqref{eq:pOdecomp}).

Recalling the definition of $\Lambda_N$ (see \eqref{eqn6.6}) we define a collection of cubes
\begin{equation}\label{eqn6.7}
 \mathcal{D}_N = \{Q\in\mathbb{D}: Q\subset \Lambda_N \} = \Big\{Q\in\mathbb{D}: Q \bigcap \bigcup_{B_l\in \mathcal{F}_N} B_l = \emptyset\Big\}. 
 \end{equation}
Note that 
\begin{itemize}
	\item $\mathcal{D}_N$ is a collection of ``good cubes''  for $N$, that is, 
		\begin{equation}\label{eq:asst}
			\frac{1}{N} \leq \frac{\omega(Q)}{\sigma(Q)} \leq N, \quad \text{for all } Q\in \mathcal{D}_{N}.
		\end{equation}
	\item If $Q\in \mathcal{D}_N$ is a ``good cube'', all of its descendants are ``good cubes'' in $\mathcal{D}_N$.
	\item The set $\Lambda_N = \cup_{Q\in\mathcal{D}_N} Q$ can be decomposed into a countable union of disjoint cubes in $\mathcal{D}_N$ with diameter less then $\delta(X_0)/4$.
\end{itemize}

Let $Q_0 \in \mathcal{D}_N$ be such that $4 \diam Q_0 \leq \delta(X_0) $.
For any descendant $Q$ of $Q_0$ (thus $Q \in \mathcal{D}_N$), by \eqref{eq:asst} we have
\begin{equation}\label{eq:Ainftycube}
\frac{1}{N^2} \frac{\omega(Q_0)}{\sigma(Q_0)} \leq \frac{\omega(Q)}{\sigma(Q)} \leq N^2 \frac{\omega(Q_0)}{\sigma(Q_0)}\qquad\hbox{     and     }\qquad	\frac{1}{N^2} \frac{\sigma(Q)}{\sigma(Q_0)} \leq \frac{\omega(Q)}{\omega(Q_0)} \leq N^2 \frac{\sigma(Q)}{\sigma(Q_0)}.
\end{equation}
To show that \eqref{eq:AinftyD} holds, the next step is to prove that \eqref{eq:Ainftycube} holds if we replace the dyadic cube $Q$ by any Borel set $E\subset Q_0$. The argument is similar to the one used in the proof of the change of pole formula \eqref{eq:copBorel}, except that now we need to work with dyadic ``cubes'' instead of surface balls.
Suppose $V \subset Q_0$ is (relatively) open. For any $x\in V$ let $Q_x$ be a dyadic cube containing $x$ such that 
\[ Q_x \subset \Delta(c_x, C_1 r_x) \subset V . \]
Here $C_1$, $c_x\in Q_x$ and $r_x = 2^{-k_x}$ are such that properties $(iv)$ and $(v)$ of Lemma \ref{lm:dgrid} hold. In particular 
$\diam Q_x \leq C_1 r_x$, and  $Q_x$ contains some surface ball $\Delta(c_x, a_0 r_x)$.
Then $V \subset \cup_{x\in V} \Delta(c_x,C_1 r_x)$. By Vitali covering lemma there is a countable collection of pairwise disjoint balls $\{\Delta(c_{x_j}, C_1 r_{x_j}) \}_{j\in J}$ such that
\begin{equation}\label{eq:coverbycube}
	V \subset \bigcup_{j\in J}  \Delta\left(c_{x_j},5C_1 r_{x_j}\right).
\end{equation}
By \eqref{eq:Ainftycube}, \eqref{eq:coverbycube}, the doubling property of $\omega$ and the fact that $\Delta(c_x,a_0 r_x) \subset Q_x$, we have
\begin{align}\label{eq:Ainftyopen}
	\frac{\omega(V)}{\omega(Q_0)}& \leq  \sum_{j\in J} \frac{\omega\left(\Delta(c_{x_j},5C_1 r_{x_j})\right)}{\omega(Q_0)} \le C \sum_{j\in J} \frac{\omega\left(\Delta(c_{x_j},a_0 r_{x_j})\right)}{\omega(Q_0)}   \\
	& \leq C\sum_{j\in J} \frac{\omega \left(Q_{x_j}\right)}{\omega(Q_0)} 
	 \leq C N^2 \sum_{j\in J} \frac{\sigma(Q_{x_j})}{\sigma(Q_0)} \nonumber \\
	& \leq C N^2 \sum_{j\in J} \frac{\sigma(\Delta(c_{x_j},C_1 r_{x_j}))}{\sigma(Q_0)} 
	 =C N^2 \frac{\sigma\left( \bigcup_{j\in J} \Delta(c_{x_j},C_1 r_{x_j}) \right)}{\sigma(Q)} \nonumber \\
	& \leq C N^2 \frac{\sigma(V)}{\sigma(Q_0)}.\nonumber
\end{align}
Since $\sigma$ is a Radon measure, \eqref{eq:Ainftyopen} holds if we replace open set $V $ by any Borel set $E\subset Q_0$ (see proof of \eqref{eqn6.4}). That is 
\begin{equation}\label{eq:AinftyBorel}
	\frac{\omega(E)}{\omega(Q_0)} \le C N^2 \frac{\sigma(E)}{\sigma(Q_0)},
\end{equation}
where $C$ only depends on $C_1$, $a_0$ and the doubling constant of $\omega$, and which in turn only depend on the depend on $n$, the Ahlfors regular constant of $\sigma$ and the uniform character of $\Omega$.
Since $\sigma$ is Ahlfors regular, it is also a doubling Radon measure. Noting that \eqref{eq:asst} and \eqref{eq:Ainftycube} are symmetric in $\sigma$ and $\omega$. By reversing their roles in \eqref{eq:Ainftyopen} and 
\eqref{eq:AinftyBorel} we obtain that for any Borel set $E\subset Q_0$
\begin{equation}\label{eq:AinftyBorel-2sides}
	C^{-1}\frac{1}{N^2}\frac{\sigma(E)}{\sigma(Q_0)}\le \frac{\omega(E)}{\omega(Q_0)} \le C N^2 \frac{\sigma(E)}{\sigma(Q_0)}.
\end{equation}
Given a surface ball $\Delta' \subset Q_0$ and a Borel set $E\subset \Delta'$,
combining \eqref{eq:AinftyBorel} with the left hand side of \eqref{eq:AinftyBorel-2sides} applied to $\Delta'$ we obtain
\begin{equation}\label{eq:localAinfty}
	\frac{\omega(E)}{\omega(\Delta')} \le C N^4 \frac{\sigma(E)}{\sigma(\Delta')}.
\end{equation}

For  $\epsilon$ small enough, define 
\[
	Q_0^*(\epsilon) = \{q\in Q_0: \dist(q,Q_0^c) > \tau r_{Q_0}\}.
\]
Note that $Q_0^*(\epsilon)$ is open. Here $r_{Q_0} = 2^{-k}$ for some $k\in\mathbb{Z}$, $\tau=\tau(\epsilon)$ is in $(0,a_0)$ such that $C_1 \tau^\eta \leq \epsilon$, and both parameters are to guarantee that properties $(iv)$ and $(v)$ in Lemma \ref{lm:dgrid} hold.
Thus we have
\[
	\sigma(Q_0^*(\epsilon)) \geq \sigma(Q_0) - C_1 \tau^{\eta}\sigma(Q_0) \geq (1-\epsilon) \sigma(Q_0).
\]
Therefore for any sequence $\epsilon_i \to 0$ we have $\sigma(Q_0 \setminus \cup_i Q_0^*(\epsilon_i)) = 0$. Thus in particular
\begin{equation}\label{eqn6.8}
Q_0= \mathcal{E}_0\cup \bigcup_i Q_0^*(\epsilon_i)\qquad\hbox{ with }\qquad \sigma(\mathcal {E}_0)=0.
\end{equation}
Thus to show $Q_0$ is rectifiable, it suffices to show $Q_0^*(\epsilon)$ is rectifiable for $\epsilon$ small enough. 
We finish the proof by applying Theorem \ref{thm:lfplocal} to the open set $Q_0^*(\epsilon)$. Suppose $\Delta'\subset \Delta$ are surface balls in $Q_0^*(\epsilon)$, and that $E\subset \Delta'\subset \Delta$ is a Borel set. Recall that  $4 \diam Q_0 \leq \delta(X_0) $ so by the change of pole formula \eqref{eq:cop} and \eqref{eq:copBorel} we have
\begin{equation}\label{eq:copapply}
	\omega^{A_{\Delta}}(\Delta') \sim \frac{\omega(\Delta')}{\omega(\Delta)},\quad \omega^{A_{\Delta}}(E) \sim \frac{\omega(E)}{\omega(\Delta)}.
\end{equation}
Combining \eqref{eq:copapply} and \eqref{eq:localAinfty} we get
\begin{align*}\label{eqn6.9}
	 \frac{\omega^{A_\Delta}(E)}{\omega^{A_{\Delta}}(\Delta')} \sim \dfrac{\frac{\omega(E)}{\omega(\Delta)}}{\frac{\omega(\Delta')}{\omega(\Delta)}} = \frac{\omega(E)}{\omega(\Delta')} C \le N^4 \frac{\sigma(E)}{\sigma(\Delta')}.
\end{align*}
That is to say $Q_0^*(\epsilon)$ satisfies the assumption \eqref{eq:AinftyD} of Theorem \ref{thm:lfplocal} with uniform constants $C_0= C N^4$ and $\theta = 1$. Therefore we conclude that $Q_0^*(\epsilon)$ is $(n-1)$-rectifiable, and using \eqref{eqn6.8} we also have that $Q_0$ is $(n-1)$-rectifiable.
By \eqref{eq:pOdecomp} $\displaystyle \pO=R_0\cup \cup_{N\ge N_0}\Lambda_N$ with $\sigma(R_0)=0$. Since each $\Lambda_N$ can be written as a countable disjoint union of cubes in $\mathcal D_N$ with diameter less than $\delta(X_0)/4$ (see \eqref{eqn6.7}) and the properties stated thereafter) such as $Q_0$, we deduce that each $\Lambda_N$ is $(n-1)$-rectifiable and so is  $\pO$.\qed


\end{myproof}

\end{document}